\documentclass[12pt,a4]{amsart}
\usepackage[a4paper, left=28mm, right=28mm, top=28mm, bottom=34mm]{geometry}
\usepackage[all]{xy}
\usepackage{amsmath}
\usepackage{amssymb}
\usepackage{amsthm}
\usepackage{mathrsfs}
\usepackage{mathtools}
\usepackage{makecell}
\usepackage{hyperref}
\usepackage{breakurl}
\usepackage{color}
\theoremstyle{definition}
\newtheorem{thm}{Theorem}[section]
\newtheorem{lem}[thm]{Lemma}
\newtheorem{cor}[thm]{Corollary}
\newtheorem{prop}[thm]{Proposition}
\newtheorem{conj}[thm]{Conjecture}
\theoremstyle{definition}

\newtheorem{rem}[thm]{Remark}

\newtheorem{ex}[thm]{Example}

\def\F{{\mathbb F}}

\def\PP{{\mathbb P}}
\def\Q{{\mathbb Q}}
\def\R{{\mathbb R}}
\def\Z{{\mathbb Z}}
\def\C{{\mathbb C}}

\def\Gal{\mathop{\mathrm{Gal}}}
\def\GL{{\mathop{\mathrm{GL}}}}
\def\Pic{{\mathop{\mathrm{Pic}}}}
\def\Jac{{\mathop{\mathrm{Jac}}}}
\def\PSL{{\mathop{\mathrm{PSL}}}}
\def\SL{{\mathop{\mathrm{SL}}}}

\def\Supp{{\mathop{\mathrm{Supp}}}}
\def\New{\mathrm{New}}
\def\new{\mathrm{new}}

\numberwithin{equation}{section}
\begin{document}

\title[Modularity of elliptic curves]{
The modularity of elliptic curves over all but finitely many totally real fields of degree $5$}

\date{October 9, 2022}

\author{Yasuhiro Ishitsuka}
\address{
Institute of Mathematics for Industry, Kyushu University,
Fukuoka, 819--0395, Japan}
\email{yishi1093@gmail.com}

\author{Tetsushi Ito}
\address{
Department of Mathematics, Faculty of Science, Kyoto University
Kyoto, 606--8502, Japan}
\email{tetsushi@math.kyoto-u.ac.jp}

\author{Sho Yoshikawa}
\address{
Gakushuin University, Department of Mathematics, 1-5-1, Mejiro, Toshima-ku, Tokyo, 171--8588, Japan}
\email{yoshikawa@math.gakushuin.ac.jp}


\maketitle

\begin{abstract}
We study the finiteness of low degree points on
certain modular curves and their Atkin--Lehner quotients,
and, as an application, prove the modularity of elliptic curves
over all but finitely many totally real fields of degree $5$.
On the way, we prove a criterion for the finiteness of
rational points of degree $5$ on a curve of large genus
over a number field using the results of Abramovich--Harris
and Faltings on subvarieties of Jacobians.
\end{abstract}

\section{Introduction}

Let $F$ be a totally real field,
i.e., $F$ is a finite extension of $\Q$ such that
$\iota(F) \subset \R$ for every embedding
$\iota \colon F \hookrightarrow \C$.
An elliptic curve $E$ over $F$ is said to be \textit{modular}
if its Hasse--Weil $L$-function coincides with
the $L$-function of a Hilbert modular newform of
parallel weight $2$ over $F$.
The following conjecture is a generalization of
the famous Shimura--Taniyama conjecture.

\begin{conj}
\label{Conjecture:Modularity}
Every elliptic curve over a totally real field $F$ is modular.
\end{conj}

Since the breakthrough of Wiles and Taylor--Wiles,
we know the modularity of many elliptic curves over
totally real fields.
Conjecture \ref{Conjecture:Modularity}
is completely proved when $[F : \Q] \leq 3$;
see \cite{Breuil-Conrad-Diamond-Taylor},
\cite{Freitas-Le Hung-Siksek}, \cite{Derickx-Najman-Siksek}.
Although Conjecture \ref{Conjecture:Modularity} is known to be true for
certain totally real fields of higher degree
(e.g.,
\cite{Yoshikawa1}, \cite{Yoshikawa2}, \cite{Thorne:Cyclotomic}),
it is still open in general when $[F : \Q] \geq 4$.
In his recent paper,
Box proved Conjecture \ref{Conjecture:Modularity}
for totally real fields of degree $4$
not containing $\sqrt{5}$; see \cite{Box:Modularity}.

There are finiteness results for
possibly non-modular elliptic curves.
Le Hung proved in his Ph.D Thesis that,
for every totally real field $F \subset \R$,
there exist only finitely many
$\overline{\Q}$-isomorphism classes of
non-modular elliptic curves defined
over totally real quadratic extensions of $F$;
see \cite[Theorem 1.1]{Le Hung:Thesis}.
Combining Le Hung's result for $F = \Q(\sqrt{5})$
with Box's results,
Conjecture \ref{Conjecture:Modularity} holds true for
all but finitely many totally real fields of degree $4$.

In this paper,
we shall study the case of totally real fields of degree $5$.
The following is the main theorem of this paper.

\begin{thm}
\label{MainTheorem}
There is a finite list of pairs
$(F_1,\alpha_1),\ldots,(F_r,\alpha_r)$,
where $F_i \subset \R$ is a totally real field of degree $5$
and $\alpha_i \in F_i$,
satisfying the following condition:
\begin{quote}
Let $E$ be an elliptic curve
over a totally real field $F \subset \R$ of degree $5$.
If $(F,j(E)) \neq (F_i, \alpha_i)$ for every $1 \leq i \leq r$,
then $E$ is modular.
\end{quote}
\end{thm}

Here $j(E)$ is the $j$-invariant of $E$.
Theorem \ref{MainTheorem} implies that
Conjecture \ref{Conjecture:Modularity} holds true
for all but finitely many totally real fields
$F \subset \R$ of degree $5$.

The strategy of our proof of Theorem \ref{MainTheorem}
is basically the same as
\cite{Freitas-Le Hung-Siksek}, \cite{Derickx-Najman-Siksek},
\cite{Box:Modularity};
see Section \ref{Section:Strategy} for details.
Thanks to recent advances on the modularity lifting theorems,
it is enough to show that
each of the following modular curves
\[
X_0(105),  \quad
X(\mathrm{s3},\mathrm{b5},\mathrm{b7}), \quad
X(\mathrm{b3},\mathrm{b5},\mathrm{e7}), \quad
X(\mathrm{s3},\mathrm{b5},\mathrm{e7}) 
\]
has only finitely many rational points of degree $5$.
These curves have rather large genera.
Their genera are $13$, $21$, $73$, $153$, respectively;
see \cite[Introduction]{Box:Modularity}.
Thus it is unlikely that we can obtain finiteness results
by explicit calculation.
Instead, we use geometric methods in this paper.

On the way, we prove a criterion for the finiteness of
rational points of degree $5$ on a curve of large genus
over a number field.

\begin{thm}[see Theorem \ref{Theorem:FinitenessLowDegreePoints2}]
\label{MainTheorem:Finiteness}
Let $C$ be a projective smooth geometrically integral curve
of genus $g(C)$ over a number field $K$.
Let $\Jac(C) \sim A_1 \times \cdots \times A_r$
be the decomposition of $\Jac(C)$, up to isogeny,
into the product of $K$-simple abelian varieties.
Assume that the following conditions are satisfied.
\begin{itemize}
\item $g(C) \geq 11$.
\item For every $i$ with $\dim A_i = 1$, the Mordell--Weil group $A_i(K)$ is finite.
\item There does not exist a non-constant morphism
$C \to \PP^1$ of degree $5$ defined over $K$.
\end{itemize}
Then the set of rational points of degree $5$ on $C$
\[ \{ x \in C(\overline{K}) \mid [K(x) : K] = 5 \} \]
is finite.
\end{thm}

Theorem \ref{MainTheorem:Finiteness} seems new.
It might be of independent interest.
We shall prove Theorem \ref{MainTheorem:Finiteness} using
the results of  Abramovich--Harris \cite{Abramovich-Harris}
on the structure of abelian varieties embedded in
the Brill--Noether loci of Jacobian varieties,
and a theorem of Faltings \cite{Faltings:Lang}
on a generalization of Mordell's conjecture (Lang's conjecture).
It is possible to generalize
Theorem \ref{MainTheorem:Finiteness}
to rational points of degree $d$, at least when $d$ is a prime number;
see Remark \ref{Remark:GenezalizationFiniteness}.
Note that we do not assume anything on $K$-simple factors of dimension $\geq 2$.
In particular, the second condition holds if $\Jac(C)$ does not
contain any elliptic curve over $K$.

In order to prove Theorem \ref{MainTheorem},
we shall apply Theorem \ref{MainTheorem:Finiteness}
to some modular curves over $\Q$.
In fact,  we shall also apply it
to an Atkin--Lehner quotient of a modular curve;
see Remark \ref{Remark:WhyAtkinLehnerquotient}.
It is not difficult to check the first and the third conditions.
We shall check the second condition in the following way.
Since all cusp forms of weight $2$ corresponding to $\Q$-simple factors
of the Jacobians in this paper have analytic rank $\leq 1$,
the rank part of the Birch and Swinnerton-Dyer conjecture is known
to be true for them; see Remark \ref{Remark:BSDknowncases}.
Thus, all we need to do is to show that,
among the cusp forms of weight $2$ on the modular curves in this paper,
all cusp forms whose Fourier coefficients are in $\Q$ have analytic rank $0$.
Fortunately, it is a routine exercise today thanks to
\textit{the L-functions and modular forms database (LMFDB)} \cite{LMFDB}.
See Section \ref{Section:Strategy} and Appendix \ref{Appendix:Calculation}
for more details.

The outline of this paper is as follows.
In Section \ref{Section:Strategy},
we explain the strategy of the proof.
In Section \ref{Section:FinitenessCriterion},
we give three theorems
(Theorem \ref{Theorem:AbramovichFrey},
Theorem \ref{Theorem:FinitenessLowDegreePoints1},
Theorem \ref{Theorem:FinitenessLowDegreePoints2}
(= Theorem \ref{MainTheorem:Finiteness}))
on the finiteness of low degree points on curves over number fields.
Theorem \ref{Theorem:AbramovichFrey} and
Theorem \ref{Theorem:FinitenessLowDegreePoints1}
are well-known.
Theorem \ref{Theorem:FinitenessLowDegreePoints2}
is proved in Section \ref{Section:FinitenessCriterion}
using the results of Abramovich--Harris and Faltings.
In Section \ref{Section:CastelnuovoSeveriInequality},
we recall the Castelnuovo--Severi inequality,
which is a main tool to study the gonality of modular curves.
The main results are proved in
Section \ref{Section:ModularCurves}
assuming some results on the modular curves.
Finally, in Appendix \ref{Appendix:Calculation},
these results are proved using LMFDB.

The datasets generated during the current study are publicly available as ancillary files at \texttt{arXiv:2110.04078}. 

\begin{rem}
It is a challenging problem to explicitly give
a finite list of ($\overline{\Q}$-isomorphism classes of)
possibly non-modular elliptic curves over a totally real field
$F \subset \R$ of degree $5$.
Finding such a list amounts to finding 
all rational points of degree $5$
on some modular curves of large genera;
see Section \ref{Section:Strategy}.
However, as well as Le Hung's results
\cite[Theorem 1.1]{Le Hung:Thesis},
our results are not effective
because the proof relies on a theorem of Faltings \cite{Faltings:Lang}.
\end{rem}

\begin{rem}
It is natural to ask whether it is possible to obtain
similar results for totally real fields of degree $6$ (or higher).
There are apparent difficulties to study rational points of degree $6$ on $X_0(105)$
because there exists a non-constant morphism
$X_0(105) \to \PP^1$ of degree $6$ over $\Q$;
see Proposition \ref{Proposition:GonalityX0(105)}.
Pulling back $\Q$-rational points,
we obtain infinitely many rational points of degree $6$ on $X_0(105)$.
It might be the case that there are infinitely many
($\overline{\Q}$-isomorphism classes of)
elliptic curves over totally real fields of degree $6$
whose associated mod $p$ Galois representations are reducible
for all $p \in \{ 3,5,7 \}$.
Currently, it seems difficult to prove the modularity of
such elliptic curves.
Therefore, without introducing new ideas or techniques,
it seems difficult to generalize
Theorem \ref{MainTheorem} to totally real fields of degree $6$.
\end{rem}

\begin{rem}
\label{Remark:BSDknowncases}
In this paper, we only consider cusp forms
of weight $2$ and with trivial character.
Using LMFDB, we see that all cusp forms
appearing in this paper have analytic rank $\leq 1$;
see Appendix \ref{Appendix:Calculation}.
Hence the rank part of the Birch and Swinnerton-Dyer conjecture
holds for abelian varieties of $\GL_2$-type
associated with them
by the theorems of Gross--Zagier and Kolyvagin--Logachev;
see
\cite[Theorem A]{Zhang:HeegnerPoints}, \cite[Theorem 3.7]{Zhang:Survey}.
In particular, if $f$ has analytic rank $0$,
the associated abelian variety $A_f$ has finite Mordell--Weil
group over $\Q$, i.e.,\ $A_f(\Q)$ is finite.
(In the case of analytic rank $0$,
another proof was given by Kato \cite[Corollary 14.3]{Kato:EulerSystem}.
Kato also proved the rank $0$ case for cusp forms of weight $2$
with non-trivial character;
see the paragraph after \cite[Corollary 14.3]{Kato:EulerSystem}.)
\end{rem}

\section{Strategy of the proof}
\label{Section:Strategy}

To prove Theorem \ref{MainTheorem},
we basically follow the same strategy as
\cite{Freitas-Le Hung-Siksek}, \cite{Derickx-Najman-Siksek},
\cite{Box:Modularity}.

We briefly explain the background of our methods.
Let $E$ be a elliptic curve over a totally real field $F$.
To simplify the discussion, we always assume
a totally real field is embedded into $\R$.
Let $\overline{\rho}_{E,p} \colon \Gal(\overline{\Q}/F) \to \GL_2(\F_p)$
be the mod $p$ Galois representation associated with
the $p$-torsion points on $E$.
Let $\zeta_p \coloneqq \exp(2 \pi \sqrt{-1}/p)$
be a primitive $p$-th root of unity.

The following results are known.
\begin{itemize}
\item
Based on the famous $(3,5)$-trick of Wiles
(and some ideas of Manoharmayum on the $(3,7)$-trick),
Freitas--Le Hung--Siksek proved that $E$ is modular if
$\overline{\rho}_{E,p}|_{\Gal(\overline{\Q}/F(\zeta_p))}$
is absolutely irreducible for some $p \in \{ 3,5,7 \}$;
see \cite[Theorem 3, Theorem 4]{Freitas-Le Hung-Siksek}.

\item Thorne proved that $E$ is modular if
$F$ does not contain $\sqrt{5}$
and $\overline{\rho}_{E,5}$ is irreducible;
see \cite[Theorem 1.1]{Thorne}.

\item Generalizing Thorne's techniques,
Kalyanswamy proved that $E$ is modular if
$F \cap \Q(\zeta_7) = \Q$,
$\overline{\rho}_{E,7}$ is irreducible,
and
$\overline{\rho}_{E,7}(\Gal(\overline{\Q}/F))$
is not conjugate to a subgroup of
$C_{\mathrm{ns}}^{+}(7)$;
see \cite[Theorem 1.2]{Kalyanswamy}.
\end{itemize}

By the above results,
if a non-modular elliptic curve $E$ over $F$ exists,
it corresponds to an $F$-rational point of
one of the following modular curves;
see \cite[Theorem 1.2]{Box:Modularity}.
\[
X(\mathrm{b3},\mathrm{b5},\mathrm{b7}) = X_0(105),  \quad
X(\mathrm{s3},\mathrm{b5},\mathrm{b7}), \quad
X(\mathrm{b3},\mathrm{b5},\mathrm{e7}), \quad
X(\mathrm{s3},\mathrm{b5},\mathrm{e7}) 
\]
Here we use the notation of
\cite{Freitas-Le Hung-Siksek}, \cite{Derickx-Najman-Siksek},
\cite{Box:Modularity}.
The symbols in the parentheses denote
the level structures defining them.
For $p \in \{ 3,5,7 \}$,
$\mathrm{b}p$ is a Borel subgroup of $\GL_2(\F_p)$,
and
$\mathrm{s}p$ is the normalizer of
a split Cartan subgroup of $\GL_2(\F_p)$.
For $p = 7$,
the symbol $\mathrm{e7}$ corresponds to
a particular subgroup of $\GL_2(\F_7)$ of order $48$
which is a subgroup of index $2$ of $C_{\mathrm{ns}}^{+}(7)$,
where $C_{\mathrm{ns}}^{+}(7)$ is
the normalizer of a non-split Cartan subgroup of $\GL_2(\F_7)$.

These modular curves have rather large genera.
Their genera are $13$, $21$, $73$, $153$, respectively.
It seems infeasible to determine the set of low degree points explicitly.
Instead, we use geometric methods to prove
the finiteness of rational points of degree $5$.

First, we study the set of $\Q$-rational points.

\begin{lem}
\label{Lemma:Cusp}
All the $\Q$-rational points on the following modular curves
\[
X_0(105), \quad
X(\mathrm{s3},\mathrm{b5},\mathrm{b7}), \quad
X(\mathrm{b3},\mathrm{b5},\mathrm{e7}), \quad
X(\mathrm{s3},\mathrm{b5},\mathrm{e7})
\]
are the cusps.
\end{lem}

\begin{proof}
We have the following finite morphisms defined over $\Q$.
\[
X_0(105) \to X(\mathrm{b5},\mathrm{b7}) = X_0(35), \quad
X(\mathrm{s3},\mathrm{b5},\mathrm{b7}) \to X_0(35).
\]
Similarly, we have the following finite morphisms defined over $\Q$.
\[
X(\mathrm{b3},\mathrm{b5},\mathrm{e7})
\to X(\mathrm{b5},\mathrm{e7}), \quad
X(\mathrm{s3},\mathrm{b5},\mathrm{e7})
\to X(\mathrm{b5},\mathrm{e7})
\]
We also have a morphism
$X(\mathrm{b5},\mathrm{e7})
\to X(\mathrm{b5},\mathrm{ns7})$
of degree $2$ defined over $\Q$.
Here, $\mathrm{ns7}$ means the level structure at $7$ is
given by $C_{\mathrm{ns}}^{+}(7)$.
It is well-known that
all $\Q$-rational points on $X_0(35)$ and
$X(\mathrm{b5},\mathrm{ns7})$ are the cusps;
see \cite[Theorem 6, Theorem 7]{Derickx-Najman-Siksek}
for example.
Hence the same assertion holds true for
$X_0(105)$,
$X(\mathrm{s3},\mathrm{b5},\mathrm{b7})$,
$X(\mathrm{b3},\mathrm{b5},\mathrm{e7})$,
$X(\mathrm{s3},\mathrm{b5},\mathrm{e7})$.
\end{proof}

In the following, assume $[F : \Q] = 5$.
By Lemma \ref{Lemma:Cusp},
the residue field of a point on a modular curve
corresponding to a (hypothetical) non-modular elliptic curve
$E/F$ is equal to $F$.
The field of definition of a rational point
cannot be smaller than $F$.
Here we use the fact that $[F : \Q] = 5$ is prime.
To prove Theorem \ref{MainTheorem},
it is enough to prove the finiteness of rational points of degree $5$
on these curves.
For some technical reasons,
we shall not directly work on the curve
$X(\mathrm{b3},\mathrm{b5},\mathrm{e7})$.
Instead, we will work on
$X(\mathrm{b3},\mathrm{b5},\mathrm{ns7})/w_3$;
see Remark \ref{Remark:WhyAtkinLehnerquotient}.
We have the following sequence of degree $2$ morphisms
defined over $\Q$
\[
X(\mathrm{b3},\mathrm{b5},\mathrm{e7}) \to
X(\mathrm{b3},\mathrm{b5},\mathrm{ns7}) \to
X(\mathrm{b3},\mathrm{b5},\mathrm{ns7})/w_3.
\]
Hence the finiteness of rational points of degree $5$ on
$X(\mathrm{b3},\mathrm{b5},\mathrm{ns7})/w_3$
implies the finiteness of rational points of degree $5$ on
$X(\mathrm{b3},\mathrm{b5},\mathrm{e7})$ and
$X(\mathrm{b3},\mathrm{b5},\mathrm{ns7})$.

In summary, in order to prove Theorem \ref{MainTheorem},
it is enough to prove the following theorem.

\begin{thm}
\label{MainTheorem:FinitenessLowDegreePointsModularCurves}
Let $X$ be one of the following curves:
\[
X_0(105), \quad
X(\mathrm{s3},\mathrm{b5},\mathrm{b7}), \quad
X(\mathrm{b3},\mathrm{b5},\mathrm{ns7})/w_3, \quad
X(\mathrm{s3},\mathrm{b5},\mathrm{e7}).
\]
Then the set of rational points of degree $5$ on $X$
\[ \{ x \in X(\overline{\Q}) \mid [\Q(x) : \Q] = 5 \} \]
is finite.
\end{thm}

The proof of this theorem will be given in the rest of this paper.

\section{Finiteness criteria for low degree points}
\label{Section:FinitenessCriterion}

\subsection{Statements}

In this section, we give three theorems concerning
the finiteness of rational points of low degree.
Let $C$ be a projective smooth geometrically integral
curve over a field $K$.
The genus of $C$ is denoted by $g(C)$.
The \textit{$K$-gonality} of $C$
is the smallest degree of a non-constant morphism $C \to \PP^1$
defined over $K$.

In the following, let $K$ be a number field.
Then, the $\overline{K}$-gonality is
equal to the $\C$-gonality for any embedding
$K \hookrightarrow \C$.
The $K$-gonality of $C$ is sometimes larger
than $\overline{K}$-gonality of $C$.
For $d \geq 1$, the set of rational points of degree $d$ on $C$
\[ \{ x \in C(\overline{K}) \mid [K(x) : K] = d \} \]
has been studied by several people
\cite{Harris-Silverman}, \cite{Debarre-Klassen}, \cite{Frey}.

\begin{rem}
Harris--Silverman proved that the set of rational points of degree $2$ on $C$
is finite if $C$ is neither hyperelliptic nor bielliptic
\cite[Corollary 3]{Harris-Silverman}.
This result is used to prove Le Hung's results \cite[Theorem 1.1]{Le Hung:Thesis}
mentioned in the Introduction.
\end{rem}

For general $d \geq 1$, the following result is obtained
by Abramovich and Frey; see \cite[Proposition 2]{Frey}.
As explained in  \cite[Section 2.1]{Abramovich:Gonality},
it also follows immediately from
\cite[Lemma 1]{Abramovich-Harris}.

\begin{thm}[Abramovich, Frey]
\label{Theorem:AbramovichFrey}
If the $\overline{K}$-gonality of $C$ is larger than $2d$,
then the set of rational points of degree $d$ on $C$
is finite.
\end{thm}

Theorem \ref{Theorem:AbramovichFrey}
does not say anything when the $\overline{K}$-gonality of $C$
is smaller than or equal to $2d$.

In this paper, we shall use two more theorems
on the finiteness of rational points of low degree.
The first one is easy and certainly well-known.

\begin{thm}
\label{Theorem:FinitenessLowDegreePoints1}
Assume that the Mordell--Weil group $\Jac(C)(K)$ is finite,
and there does not exist a non-constant morphism
$C \to \PP^1$ of degree $d$ defined over $K$.
Then the set of rational points of degree $d$ on $C$ is finite.
\end{thm}

In the next theorem,
we only consider rational points of degree $5$.
We need to assume the genus is large.
Instead, we relax the assumptions on the Mordell--Weil group.

\begin{thm}
\label{Theorem:FinitenessLowDegreePoints2}
Let $C$ be a projective smooth geometrically integral curve
of genus $g(C)$ over a number field $K$.
Let $\Jac(C) \sim A_1 \times \cdots \times A_r$
be the decomposition of $\Jac(C)$, up to isogeny,
into the product of $K$-simple abelian varieties.
Assume that the following conditions are satisfied.
\begin{itemize}
\item $g(C) \geq 11$.
\item For every $i$ with $\dim A_i = 1$, the Mordell--Weil group $A_i(K)$ is finite.
\item There does not exist a non-constant morphism
$C \to \PP^1$ of degree $5$ defined over $K$.
\end{itemize}
Then the set of rational points of degree $5$ on $C$
\[ \{ x \in C(\overline{K}) \mid [K(x) : K] = 5 \} \]
is finite.
\end{thm}

\begin{rem}
\label{Remark:GenezalizationFiniteness}
Although  we only consider rational points of degree $5$ in this paper,
it is certainly possible to generalize
Theorem \ref{Theorem:FinitenessLowDegreePoints2}
to rational points of degree $d$, at least when $d$ is a prime number.
For this purpose, we need to generalize
Proposition \ref{Proposition:AbelianSurfaceinJacobian} (see below)
to effective divisors of degree $d$.
Then, we expect the lower bound of the genus would be
a quadratic function of $d$.
It is related to a theorem of Abramovich--Harris
\cite[Theorem 2]{Abramovich-Harris}.
However, we warn the readers that
there are some issues on the status of this theorem;
see Remark \ref{Remark:Issue:AbramovichHarris}.
\end{rem}

\subsection{Proof of Theorem \ref{Theorem:FinitenessLowDegreePoints1}}

Theorem \ref{Theorem:FinitenessLowDegreePoints1} is well-known.
For example,
it follows from \cite[Corollary 2.4]{Derickx-Sutherland}.
Here we give a proof of it in some detail.
The lemmas given in this subsection
will also be used in the proof of
Theorem \ref{Theorem:FinitenessLowDegreePoints2} later.

Let $C$ be a projective smooth geometrically integral
curve of genus $g(C)$
over a number field $K$.
Let $x \in C(\overline{K})$ be a rational point of degree $d$.
The sum of $x$ and its conjugates defines
an effective $K$-rational divisor of degree $d$ on $C$.
We denote it by $D_x$.
Then we have a map
\[
\{ x \in C(\overline{K}) \mid [K(x) : K] = d \}
\ \to \ C^{(d)}(K), \quad
x \mapsto D_x,
\]
which is $d$-to-$1$ onto the image.
It is not necessarily surjective.
Let
\[ \Omega_d \subset C^{(d)}(K) \]
be the image of this map.
Here $C^{(d)} \coloneqq \mathrm{Sym}^d(C)$
is the $d$-th symmetric power.
We consider it as the variety of effective divisors of degree $d$.

\begin{lem}
\label{Lemma:Support}
Let $x \in C(\overline{K})$ be a rational point of degree $d$
on $C$.
Let $E$ be an effective $K$-rational divisor of degree $d$.
If $(\Supp \, D_x) \cap (\Supp \, E) \neq \emptyset$,
then $D_x = E$.
\end{lem}

\begin{proof}
The assertion follows from the fact that $D_x$
is irreducible as a $K$-rational divisor.
\end{proof}

\begin{lem}
\label{Lemma:Injectivity}
Assume that there does not exist a non-constant morphism
$C \to \PP^1$ of degree $d$ defined over $K$.
\begin{enumerate}
\item Let $x \in C(\overline{K})$ be a rational point of degree $d$.
Let $E$ be an effective $K$-rational divisor of degree $d$.
If we have an equality of divisor classes $[D_x] = [E]$, then $D_x = E$.
In other words, we have $\dim_K H^0(C, \mathscr{O}_C(D_x)) = 1$.

\item The natural map
$\Omega_d \to \Pic^d(C)$, $D_x \mapsto [D_x]$
is injective.
Here $\Pic^d(C)$ is the degree $d$ part of the Picard group of $C$.
\end{enumerate}
\end{lem}

\begin{proof}
Since (1) implies (2), it is enough to prove (1).
Let $x_i (1 \leq i \leq d)$ be the conjugates of $x$.
By assumption, the divisor $D_x \coloneqq \sum_{i=1}^{d} x_i$
is linearly equivalent to $E$.
This means there exists a non-constant morphism
$f \colon C \to \PP^1$ defined over $K$
such that $\mathrm{div}(f) = D_x - E$.
By assumption, we have $\deg f \neq d$.
Hence $(\Supp \, D_x) \cap (\Supp \, E) \neq \emptyset$.
By Lemma \ref{Lemma:Support}, we have $D_x = E$.
\end{proof}

\begin{rem}
\label{Remark:Uniqueness}
Let $x \in C(\overline{K})$ be a rational point of degree $d$,
and $D_x$ the divisor of degree $d$ associated with $x$.
Let $L/K$ be an extension of fields, e.g., $L = \overline{K}$ or $\C$.
Let $C_L$ be the base change of $C$ to $L$.
Then we have
\[
\dim_L H^0(C_L, \mathscr{O}_{C_L}(D_x))
= \dim_K H^0(C, \mathscr{O}_C(D_x)) = 1
\]
by Lemma \ref{Lemma:Injectivity} (1).
Therefore,
every effective $L$-rational divisor on $C_L$ linearly equivalent to $D_x$
is equal to $D_x$.
\end{rem}

We shall give a proof of Theorem \ref{Theorem:FinitenessLowDegreePoints1}.
If $C$ has no rational point of degree $d$,
then there is nothing to prove.
Thus we assume $C$ has a rational point of degree $d$.
We fix one such point and denote it by $x_0 \in C(\overline{K})$.
Let $D_{x_0}$ be the $K$-rational divisor of degree $d$
corresponding to $x_0$.
It gives an isomorphism of $K$-schemes
\[
\Pic^d_{C/K} \cong \Pic^0_{C/K} = \Jac(C),
\quad
[D] \mapsto [D - D_{x_0}].
\]
Here $\Pic^d_{C/K}$ is the degree $d$ part of the Picard scheme.
(For basic properties of the Picard scheme,
see \cite[Chapter 8]{Bosch-Luetkebohmert-Raynaud}.)

Since we are assuming $\Jac(C)(K)$ is finite,
$\Pic^d_{C/K}(K)$ is finite.
The natural map $\Pic^d(C) \to \Pic^d_{C/K}(K)$
sending a divisor class of degree $d$ on $C$
to its associated rational point on the Picard scheme
is always injective;
see \cite[Chapter 8, Section 8.1, Proposition 4]{Bosch-Luetkebohmert-Raynaud} for details.
Thus, $\Pic^d(C)$ is finite.
By Lemma \ref{Lemma:Injectivity} (2),
we conclude that $\Omega_d$ is finite.

The proof of Theorem \ref{Theorem:FinitenessLowDegreePoints1} is complete.
\qed

\subsection{Subvarieties of abelian varieties}
\label{Subsection:SubvarietiesAbelianVarieties}

Before giving the proof of 
Theorem \ref{Theorem:FinitenessLowDegreePoints2},
we recall some results on subvarieties the Brill--Noether loci of
Jacobian varieties
due to Abramovich--Harris \cite{Abramovich-Harris}.

\begin{rem}
\label{Remark:Issue:AbramovichHarris}
There are some issues on some of the statements and
the proofs given in \cite{Abramovich-Harris};
see \cite[p.\ 236]{Debarre-Fahlaoui},
\cite[Section 9]{vanderGeer:points},
\cite[pp.\ 3026--3027]{Tucker}.
To the authors' best knowledge,
the current situation is as follows.
\begin{itemize}
\item \cite[Lemma 6]{Abramovich-Harris} is incorrect.
Hence we should avoid using \cite[Lemma 6]{Abramovich-Harris} and its corollaries.

\item
As pointed out in \cite[p.\ 116]{vanderGeer:points},
there is a misprint in the statement of \cite[Lemma 7]{Abramovich-Harris}.
The correct statement is:
$r_{k+1} - r_{k} \geq r_{k} - r_{k-1}$.

\item The second part of \cite[Lemma 8]{Abramovich-Harris}
and \cite[Corollary 1]{Abramovich-Harris}
should be considered unproved
because their proof relies on \cite[Lemma 6]{Abramovich-Harris}.

\item It is fair to say that
\cite[Theorem 2]{Abramovich-Harris} is not proved in \cite{Abramovich-Harris}
because its proof in \cite{Abramovich-Harris} relies on
the second part of \cite[Lemma 8]{Abramovich-Harris},
which relies on \cite[Lemma 6]{Abramovich-Harris}.

\item
The current status of \cite[Theorem 2]{Abramovich-Harris}
is rather unclear.
According to 
\cite[p.\ 116]{vanderGeer:points}, \cite[p.\ 236]{Debarre-Fahlaoui},
Abramovich gave a different proof of
\cite[Theorem 2]{Abramovich-Harris} with an extra hypothesis,
but his result has not yet been published.
\end{itemize}
Although we quote \cite{Abramovich-Harris},
we shall not use any of the following results:
\cite[Lemma 6]{Abramovich-Harris},
the second part of \cite[Lemma 8]{Abramovich-Harris},
\cite[Corollary 1]{Abramovich-Harris},
and \cite[Theorem 2]{Abramovich-Harris}.
We shall only use the following results:
\cite[Lemma 1]{Abramovich-Harris},
\cite[Lemma 3]{Abramovich-Harris},
\cite[Lemma 5]{Abramovich-Harris},
and the first part of \cite[Lemma 8]{Abramovich-Harris};
it can be checked easily that none of them relies on
\cite[Lemma 6]{Abramovich-Harris}.
\end{rem}

For simplicity, we work over $\C$ in this section.
We do not distinguish algebraic varieties over $\C$
and the sets of $\C$-rational points.
Let $C$ be a projective smooth integral curve of genus $g(C)$
over $\C$.
The image of the natural map
$C^{(d)} \to \Pic^d(C)$ is denoted by $W_d(C)$.
The set of $\C$-rational points on $W_d(C)$
is identified with the set of linear equivalence classes
of effective divisors of degree $d$ on $C$.
For any $r \geq 0$,
we put
\[
W^r_d(C) \coloneqq \{ [D] \in \Pic^d(C) \mid
h^0(\mathscr{O}_C(D)) \geq r+1 \}.
\]
Then we have $W_d(C) = W^0_d(C)$,
and $W^r_d(C)$ is a closed subvariety of $W_d(C)$.
In the following, we consider the case $d = 5$.

\begin{prop}
\label{Proposition:AbelianSurfaceinJacobian}
Let $A \subset \Jac(C)$ be an abelian subvariety of dimension $\geq 2$.
Let $D$ be a $\C$-rational divisor of degree $5$ satisfying
\[ [D] + A \ \subset \ W_5(C). \]
Assume $g(C) \geq 11$.
Then one of the following holds:
\begin{itemize}
\item There exists a $\C$-rational point $x \in C$ such that $[D] + A$ is contained in $x + W_4(C)$.
\item $[D] + A$ is contained in $\Delta$.
Here $\Delta$ is the image of the \textit{big diagonal} of $C^{(5)}$.
(This means every element of $[D] + A$ is represented by a non-reduced divisor,
i.e., a divisor which has multiplicity $\ge 2$ at some point.)
\end{itemize}
\end{prop}

\begin{proof}
We prove this proposition along the same lines as
\cite[Corollary 3.6]{Debarre-Fahlaoui},
\cite[Proposition 3.8]{Debarre-Fahlaoui}.
We assume that both of the assumptions in
the beginning of \cite[Section 2]{Abramovich-Harris} are satisfied.
Namely, we assume that
the following conditions are satisfied:
\begin{itemize}
\item 
$[D] + A$ is not contained in $x + W_4(C)$
for any $\C$-rational point $x \in C$,

\item 
$[D] + A$ is not contained in the image $\Delta$ of the big diagonal of $C^{(5)}$.
\end{itemize}
We shall derive a contradiction from these assumptions.

We shall briefly recall the notation used in \cite[Section 2]{Abramovich-Harris}. 
For $k \geq 1$,
let $A_k \subset W_{5k}(C)$ be the image of the summation map
\[
  \underbrace{([D] + A) \times \cdots \times ([D] + A)}_{k}
  \ \hookrightarrow \ \underbrace{W_5(C) \times \cdots \times W_5(C)}_{k}
  \ \overset{\mathrm{sum}}{\to} \ W_{5k}(C).
\]
For a divisor class $\alpha \in A_k$,
let $D_{\alpha}$ be an effective divisor of degree $5k$ associated with $\alpha$.
Let
\[ r(\alpha) \coloneqq \dim |D_{\alpha}| \]
be the dimension of the complete linear system associated with $D_{\alpha}$.
We put
\[ r_k \coloneqq \min \{ r(\alpha) \mid \alpha \in A_k \}. \]
We have $r_1 = 0$ by the first assumption.
For a general $\alpha \in A_k$, we have $r(\alpha) = r_k$.

In the following, we shall consider the case $k = 2$.
By the first assumption, for a general $\alpha \in A_2$,
the complete linear system $|D_{\alpha}|$ is base point free.
Let
$\phi_{\alpha} \colon C \to \PP^{r(\alpha)}$
be the morphism associated with $|D_{\alpha}|$.

We shall show that $\phi_{\alpha}$ is not birational onto
its image for a general $\alpha \in A_2$.
To see this, we assume that 
$\phi_{\alpha}$ is birational onto its image for a general $\alpha \in A_2$.
By \cite[Lemma 1]{Abramovich-Harris}, we always have
$r_2 \geq \dim A \geq 2$,
and under our birationality assumption,
by \cite[Lemma 5]{Abramovich-Harris},
we further have $r_2 \geq 3$.
By the first claim in the first part of \cite[Lemma 8]{Abramovich-Harris},
we have $r_3 \geq 2 r_2 \geq 6$.
This follows from the corrected statement
``$r_{k+1} - r_{k} \geq r_{k} - r_{k-1}$''
of \cite[Lemma 7]{Abramovich-Harris} for $k = 1$.
By the second claim in the first part of \cite[Lemma 8]{Abramovich-Harris} for $k = 3$,
we have
\[
  r_5 - r_3 \ \geq \ \min \{ r_3 - r_1 + r_2, \, 10 \}.
\]
Hence we have $r_5 \geq 2 r_3 + r_2 \geq 15$.
For any $\beta \in A_5$, we have $\deg \beta = 25$.
Since
\[ r(\beta) \ \geq \ r_5 \ > \ (\deg \beta)/2 \ = \ 25/2, \]
the divisor $\beta$ is non-special by Clifford's theorem \cite[Chapter III, Section 1]{ACGH}.
By Riemann--Roch, we have
\[ g(C) \ = \ \deg \beta - r(\beta) \ \leq \ 25 - 15 \ = \ 10, \]
which contradicts the assumption $g(C) \geq 11$.
The contradiction shows that
$\phi_{\alpha}$ is not birational onto its image for
a general $\alpha \in A_2$.

Now we shall apply \cite[Lemma 3]{Abramovich-Harris}.
We note that there are corrections and typos \cite[Lemma 3]{Abramovich-Harris};
see Remark \ref{Remark:AbramovichHarrisLemma3} below.
As a conclusion,
there exists a non-constant morphism $\rho \colon C \to C'$
of degree $\deg \rho > 1$ such that
\[ [D] + A \ \subset \ \rho^{\ast} W_{d'}(C')  \]
for $d' = d/\deg \rho$.
Since $d = 5$ is prime, we necessarily have $d' = 1$ and $\deg \rho = 5$.
Since
\[ \dim (\rho^{\ast} W_1(C')) = \dim W_1(C') \leq 1, \]
the variety $\rho^{\ast} W_1(C')$ cannot contain
a translate of an abelian variety of dimension $\geq 2$.

The contradiction shows that at least one of the assumptions in
\cite[Section 2]{Abramovich-Harris} cannot be satisfied.
The proof of Proposition \ref{Proposition:AbelianSurfaceinJacobian}
is complete.
\end{proof}

\begin{rem}
\label{Remark:AbramovichHarrisLemma3}
There are typos
in \cite[Lemma 3]{Abramovich-Harris}.
The equation in the last line of the statement
should read
$d' = d/\deg \rho = (\deg \overline{\phi}_{\alpha})/2$,
not
$d' = d/\deg \rho = \deg \overline{\phi}_{\alpha}$,
because $\deg \alpha = 2d$.
The equation in the fourth line of the proof of
\cite[Lemma 3]{Abramovich-Harris}
should read $d' = 2d/m$, not $d' = d/m$,
because $\deg \alpha = 2d$.

We also remark that one of the conclusions of
\cite[Lemma 3]{Abramovich-Harris} never occurs.
More precisely,
the first case ``$A \subset W^1_{d'}(C)$''
in \cite[Lemma 3]{Abramovich-Harris} never occurs.
The statement of \cite[Lemma 3]{Abramovich-Harris}
should be as follows:
\begin{quote}
Assume $\dim A \geq 1$ and $r(\alpha) > 1$ for all $\alpha \in A_2$.
If $\phi_{\alpha} \colon C \to \PP^{r(\alpha)}$
is not birational for general $\alpha$,
then $\phi_{\alpha}$ factors as
\[
C \overset{\rho}{\longrightarrow} C'
\overset{\overline{\phi}_{\alpha}}{\longrightarrow} \PP^{r(\alpha)}
\]
such that $C'$ is a projective smooth curve of genus $\geq 1$
and there is an embedding
$A \hookrightarrow \rho^{\ast} W_{d'}(C')$
where $d' = d/\deg \rho$.
\end{quote}
In the proof of \cite[Lemma 3]{Abramovich-Harris},
the case ``$A \subset W^1_{d'}(C)$''
corresponds to the case where
the image of $\phi_{\alpha}$ is a rational curve.
In the following, we shall show that this case never occurs.

Assume that the image of $\phi_{\alpha}$ is a rational curve.
The image is a rational normal curve of degree $m = r(\alpha)$.
We obtain an inclusion $i \colon A_2 \hookrightarrow W^1_{2d/m}(C)$ 
by sending $\alpha \in A_2$ to the divisor class 
corresponding to the preimage of a point of $\mathrm{Im}(\phi_\alpha)$ under $\phi_\alpha$.
By the definition of this inclusion map, the image of the summation map
\[
A_2^m \coloneqq \underbrace{A_2 \times \cdots \times A_2}_{m}
\ \overset{i^m}{\hookrightarrow} \ W_{2d/m}(C)^m \coloneqq \underbrace{W_{2d/m}(C) \times \cdots \times W_{2d/m}(C)}_{m}
\ \overset{\mathrm{sum}}{\to} \ W_{2d}(C)
\]
is $A_2$.
Since $i(A_2) \subset W^1_{2d/m}(C)$ and $A_2 \cong A$,
the preimage of $i(A_2)$ under $C^{(2d/m)} \to W_{2d/m}(C)$
has dimension at least $\dim A_2 + 1 = \dim A + 1$.
Therefore, the preimage of $i(A_2)^m$ under
$(C^{(2d/m)})^m \to W_{2d/m}(C)^m$
has dimension at least $m \cdot \dim A + m$.
We conclude that
the complete linear system corresponding to a general point of $A_2$
has dimension at least $(m-1) \cdot \dim A + m$.
Hence we have
\[ m = r(\alpha) \ \geq \ (m-1) \cdot \dim A + m. \]
Consequently, we have $m = r(\alpha) = 1$.
It contradicts the assumption
$r(\alpha) \geq 2$ of \cite[Lemma 3]{Abramovich-Harris}.
The contradiction shows that
the image of $\phi_{\alpha}$ is never a rational curve.

The authors are grateful to the referees for
suggestions for the proof of
Proposition \ref{Proposition:AbelianSurfaceinJacobian}.
The referees pointed out that our original proof of
Proposition \ref{Proposition:AbelianSurfaceinJacobian}
shows that the first case ``$A \subset W^1_{d'}(C)$''
in \cite[Lemma 3]{Abramovich-Harris} never occurs.
\end{rem}

\begin{cor}
\label{Corollary:IntersectionSupport}
Let $C, A, D$ satisfy
the assumption of Proposition \ref{Proposition:AbelianSurfaceinJacobian}.
Assume $[D] + A \not\subset \Delta$.
Let $D_1, D_2$ be effective $\C$-rational divisors of degree $5$
satisfying $[D_1], [D_2] \in [D] + A$.
Then there exist effective $\C$-rational divisors $D'_1, D'_2$
satisfying the following
\[
D_1 \sim D'_1, \quad
D_2 \sim D'_2, \quad
(\Supp \, D'_1) \cap (\Supp \, D'_2) \neq \emptyset.
\]
Here $\sim$ denotes the linear equivalence.
\end{cor}

\begin{proof}
By Proposition \ref{Proposition:AbelianSurfaceinJacobian},
we see that $[D] + A$
is contained in $x + W_4(C)$ for some $x \in C$.
Thus there exists a linear equivalence
$D_1 \sim x + E_1$ (resp.\ $D_2 \sim x + E_2$)
for an effective divisor $E_1$ (resp.\ $E_2$) of degree $4$.
We put $D'_1 \coloneqq x + E_1$ (resp.\ $D'_2 \coloneqq x + E_2$).
Then we have
$x \in (\Supp \, D'_1) \cap (\Supp \, D'_2)$.
\end{proof}

\subsection{Proof of Theorem \ref{Theorem:FinitenessLowDegreePoints2}}

We may assume $C$ has a rational point $x_0 \in C(\overline{K})$ of degree $5$.
Let $D_{x_0}$ be the $K$-rational divisor of degree $5$
corresponding to $x_0$.
It gives an isomorphism of $K$-schemes
$\Pic^5_{C/K} \cong \Jac(C)$.
In the following, we identify $\Pic^5_{C/K}$ and $\Jac(C)$
via this isomorphism.

We have a morphism of $K$-schemes
\[
\psi \colon C^{(5)} \to \Pic^5_{C/K} \cong \Jac(C).
\]
Let
$\Omega_5 \subset C^{(5)}(K)$
be the image of the set of rational points of degree $5$ on $C$.
Since the restriction
\[
\psi|_{\Omega_5} \colon \Omega_5 \to \Jac(C)(K)
\]
is injective
by Lemma \ref{Lemma:Injectivity} (2),
it is enough to show that the image $\psi(\Omega_5)$ is finite.

We now apply a theorem of Faltings \cite[Theorem 4.2]{Faltings:Lang}
to the subvariety
\[ W_5(C) \cong \psi(C^{(5)}) \subset \Jac(C). \]
There exist $K$-rational points $x_1,\ldots,x_r \in \Jac(C)(K)$
and $K$-rational abelian subvarieties
$B_1,\ldots,B_r \subset \Jac(C)$
of dimension $\geq 0$ satisfying the following conditions:
\begin{itemize}
\item $x_i + B_i \subset \psi(C^{(5)})$ for every $i$.
\item Every $K$-rational point on $\psi(C^{(5)})$
is contained in $x_i + B_i$ for some $i$.
Namely, we have
\[
\psi(C^{(5)})(K) \ = \ \bigcup_{i=1}^{r} (x_i + B_i(K)).
\]
\end{itemize}
Since $\psi(\Omega_5) \subset \psi(C^{(5)})(K)$,
it is enough to prove the finiteness of the intersection
\[
\psi(\Omega_5) \cap (x_i + B_i(K))
\]
for every $i$ with $1 \leq i \leq r$.

Recall that  we have a decomposition
$\Jac(C) \sim A_1 \times \cdots \times A_r$
into $K$-simple abelian varieties, up to isogeny.

\begin{itemize}
\item 
There is nothing to prove when $\dim B_i = 0$, i.e., $B_i$ is a point.

\item 
Assume $\dim B_i = 1$.
The elliptic curve $B_i$ is $K$-isogenous to one of $A_1,\ldots,A_r$.
Since we are assuming
$A_i(K)$ is finite if $\dim A_i = 1$,
we see that $B_i(K)$ is finite.
Hence $\psi(\Omega_5) \cap (x_i + B_i(K))$ is finite.

\item
Assume $\dim B_i \geq 2$.
In this case, we shall show the intersection
$\psi(\Omega_5) \cap (x_i + B_i(K))$
contains at most one element.
Let
$x,y \in C(\overline{K})$ be rational points of degree $5$
such that both $\psi(D_x)$ and $\psi(D_y)$
are contained in $x_i + B_i(K)$.
By Remark \ref{Remark:Uniqueness},
$D_x$ (resp.\ $D_y$) is the unique effective $\C$-rational divisor 
in the linear equivalence class $[D_x]$ (resp.\ $[D_y]$).
Since $D_x$ and $D_y$ are geometrically reduced,
we see that $[D_x], [D_y]$ are not contained 
in the image of the big diagonal $\Delta$.
Hence $x_i + B_i$ is not contained in the image of
the big diagonal of $C^{(5)}$.
By Corollary \ref{Corollary:IntersectionSupport},
there exist effective $\C$-rational divisors $D'_x, D'_y$ 
such that
\[
D_x \sim D'_x, \quad
D_y \sim D'_y, \quad \text{and} \quad
(\Supp \, D'_x) \cap (\Supp \, D'_y) \neq \emptyset.
\]
Again, by Remark \ref{Remark:Uniqueness},
we have $D_x = D'_x$ and $D_y = D'_y$.
Hence we have $(\Supp \, D_x) \cap (\Supp \, D_y) \neq \emptyset$.
By Lemma \ref{Lemma:Support},
we have $D_x = D_y$,
and the $\overline{K}$-rational points $x,y$
are conjugate to each other.
Hence we have
\[ |\psi(\Omega_5) \cap (x_i + B_i(K))| \leq 1. \]
\end{itemize}

The proof of Theorem \ref{Theorem:FinitenessLowDegreePoints2}
is complete.
\qed

\begin{rem}
\label{Remark:ExampleCurve}
We only prove the finiteness of $\psi(\Omega_5)$.
We do not prove the finiteness of $W_5(C)(K)$.
In fact, under the assumption of
Theorem \ref{Theorem:FinitenessLowDegreePoints2},
the variety $W_5(C)$ can have infinitely many $K$-rational points.
One can construct such curves by the following way.
Let $\pi \colon C \to B$ be a double covering of
projective smooth curves over $\overline{\Q}$ satisfying the following conditions:
\begin{itemize}
\item $g(C) \geq 11$.
\item $g(B) = 2$.
\item There does not exist a non-constant morphism
$C \to \PP^1$ of degree $5$ over $\overline{\Q}$.
\item $\Jac(C)$ does not contain an elliptic curve defined over $\overline{\Q}$.
\end{itemize}
For the construction of such curves,
see Example \ref{Example:DoubleCover}.
Then $\pi \colon C \to B$ is defined over a number field $K$
of sufficiently large degree.
Replacing $K$ by a finite extension of it,
we may assume
$C$ has a $K$-rational point, and $\Jac(B)(K)$ has positive rank.
Let $x \in C(K)$ be a $K$-rational point.
Since the map $\Pic^2_{B/K}(K) \to \Jac(B)(K)$
defined by $[D] \mapsto [D] - 2 [\pi(x)]$ is bijective,
$\Pic^2_{B/K}(K)$ has infinitely many elements.
Since $B$ has genus $2$, every divisor class in
$\Pic^2_{B/K}(K)$ is linearly equivalent to an effective divisor of degree $2$.
Therefore, every divisor class in
$[x] + \pi^{\ast} \Pic^2_{B/K}(K)$
is linearly equivalent to an effective divisor of degree $5$.
Since $\mathrm{Ker}(\pi^{\ast} \colon \Jac(B) \to \Jac(C))$ is finite,
we see that $C$ has infinitely many
$K$-rational effective divisors of degree $5$
(i.e.,\ $W_5(C)(K)$ has infinitely many elements).
However, by Theorem \ref{Theorem:FinitenessLowDegreePoints2},
only finitely many of them correspond to rational points of degree $5$.
\end{rem}

\section{The Castelnuovo--Severi inequality and its application}
\label{Section:CastelnuovoSeveriInequality}

In this section, we recall the Castelnuovo--Severi inequality,
which is a powerful tool to study the geometry of algebraic curves.

For simplicity, we work over $\C$ in this section.

\begin{thm}[The Castelnuovo--Severi inequality]
\label{Theorem:Castelnuovo-Severi}
Let $C_1, C_2, C_3$ be projective smooth integral curves.
Let $\C(C_1), \C(C_2), \C(C_3)$ be the function fields of
$C_1, C_2, C_3$, respectively.
Let
$f \colon C_1 \to C_2$ and $g \colon C_1 \to C_3$
be non-constant morphisms.
Then we consider the extension of function fields
$K(C_1)/f^{\ast} K(C_2)$ and $K(C_1)/g^{\ast} K(C_3)$.
Assume that $K(C_1)$ is the composite of
$f^{\ast} K(C_2)$ and $g^{\ast} K(C_3)$.
Then we have
\[
g(C_1) \ \leq \ g(C_2) \deg f + g(C_3) \deg g + (\deg f - 1) (\deg g - 1).
\]
\end{thm}

\begin{proof}
See \cite[Theorem 3.5]{Accola}, \cite[Chapter III, Theorem 3.11.3]{Stichtenoth}.
\end{proof}

We shall give an application to the non-existence
of non-constant morphisms to $\PP^1$.

\begin{cor}
\label{Corollary:Castelnuovo-Severi:involution}
Let $C, C'$ be projective smooth integral curves.
Let $\pi \colon C \to C'$ be a non-constant morphism of degree $2$.
Let $f \colon C \to \PP^1$ be a non-constant morphism
of degree $d$ which does not factor through $\pi$.
Then the following inequality holds
\[
g(C) \ \leq \ 2 g(C') + d - 1.
\]
\end{cor}

\begin{proof}
Applying Theorem \ref{Theorem:Castelnuovo-Severi}
to $f$ and $\pi$, we have
\[
g(C) \ \leq \ 0 + 2 g(C') + (\deg f - 1) (\deg \pi - 1)
= 2 g(C') + d - 1.
\]
\end{proof}

\begin{ex}
\label{Example:DoubleCover}
Here we give a construction of curves $C,B$
satisfying the conditions of Remark \ref{Remark:ExampleCurve}.
Let $B$ be a projective smooth curve of genus $2$ over $\C$.
Let $n \geq 8$.
Take distinct points
$x_1,\ldots,x_{2n} \in B(\C)$.
Let $\pi \colon C \to B$ be a double covering
ramified only at $x_1,\ldots,x_{2n}$.
By the Riemann--Hurwitz formula, we have $g(C) = n + 3 \geq 11$.
By Corollary \ref{Corollary:Castelnuovo-Severi:involution},
there does not exist a non-constant morphism
$C \to \PP^1$ of degree $5$.
On the other hand, since every genus $2$ curve is hyperelliptic,
we have a degree 2 morphism $B \to \PP^1$.
Thus we have a non-constant morphism
$C \to \PP^1$ of degree $4$.
It is well-known that
if we take $C \to B$ sufficiently general,
both $\Jac(C)/\pi^{\ast} \Jac(B)$ and $\Jac(B)$ have Picard number $1$;
see \cite[Corollary 5.3]{Biswas}.

By a specialization argument (see \cite[Theorem 1.1]{Maulik-Poonen}),
there exists a double covering $\pi \colon C \to B$
defined over a number field such that
both $\Jac(C)/\pi^{\ast} \Jac(B)$ and $\Jac(B)$
have geometric Picard number $1$.
Hence $\Jac(C)/\pi^{\ast} \Jac(B)$ and $\Jac(B)$
are $\overline{\Q}$-simple abelian varieties.
Since $\Jac(C)$ is isogenous to
$(\Jac(C)/\pi^{\ast} \Jac(B)) \times \Jac(B)$,
we see that $\Jac(C)$ does not contain an elliptic curve
over $\overline{\Q}$.
\end{ex}

\section{Applications to modular curves}
\label{Section:ModularCurves}

\subsection{The gonality of modular curves}
\label{Subsection:GonalityModularCurves}

We follow the notation of \cite{Freitas-Le Hung-Siksek}.
For the background on modular curves,
see \cite[Section 2]{Freitas-Le Hung-Siksek}.

First, we recall finite index subgroups of $\SL_2(\Z)$
corresponding to the modular curves
\[
X_0(105), \quad
X(\mathrm{s3},\mathrm{b5},\mathrm{b7}), \quad
X(\mathrm{b3},\mathrm{b5},\mathrm{ns7}), \quad
X(\mathrm{b3},\mathrm{b5},\mathrm{e7}), \quad
X(\mathrm{s3},\mathrm{b5},\mathrm{e7}).
\]
For $p \in \{ 3,5,7 \}$,
let $B(p) \subset \GL_2(\F_p)$ be a Borel subgroup,
and $C_{\mathrm{s}}^{+}(p) \subset \GL_2(\F_p)$
the normalizer of a split Cartan subgroup,
and $C_{\mathrm{ns}}^{+}(p)$
the normalizer of a non-split Cartan subgroup.
When $p = 7$, we also have another subgroup of $\GL_2(\F_7)$
defined by
\[
G(\mathrm{e7}) = \bigg\langle
\begin{pmatrix} 0 & 5 \\ 3 & 0 \end{pmatrix},\ 
\begin{pmatrix} 5 & 0 \\ 3 & 2 \end{pmatrix}
\bigg\rangle
\ \subset \ \GL_2(\F_7).
\]
Up to conjugacy, it is a subgroup of $C_{\mathrm{ns}}^{+}(7)$
of index $2$.
In the following list, we summarize the number of elements of these groups.

\vspace{0.1in}

\begin{center}
\begin{tabular}{|c|c|c|c|c|c|}
\hline
& $\GL_2(\F_p)$ & $B(p)$ & $C_{\mathrm{s}}^{+}(p)$ & $C_{\mathrm{ns}}^{+}(p)$ & $G(\mathrm{e7})$ \\
\hline
$p$ & $p (p - 1)^2 (p + 1)$ & $p (p-1)^2$ & $2 (p-1)^2$ & $2 (p^2 - 1)$ &  \\
\hline
3 & 48 & 12 & 8 & 16 &  \\
\hline
5 & 480 & 80 & 32 & 48 &  \\
\hline
7 & 2016 & 252 & 72 & 96 & 48 \\
\hline
\end{tabular}
\end{center}

\vspace{0.1in}

In the following table, we summarize the
level structure corresponding to each modular curve,
and the index of level subgroups inside
$\SL_2(\Z)$.

\vspace{0.1in}

\begin{center}
\begin{tabular}{|c|c|c|c|c|c|c|c|c|}
\hline
curve &
\makecell{level \\ at $3$} &
\makecell{level \\ at $5$} &
\makecell{level \\ at $7$}
& index & genus &
\makecell{gonality \\ (proved)} &
\makecell{gonality \\ (expected)} \\
\hline
$X_0(105)$ & $B(3)$ & $B(5)$ & $B(7)$ &
192 & 13 & $\geq 1.904$ & $\geq 2$ \\
\hline
$X(\mathrm{s3},\mathrm{b5},\mathrm{b7})$ & $C_{\mathrm{s}}^{+}(3)$ & $B(5)$ & $B(7)$ &
288 & 21 & $\geq 2.856$ & $\geq 3$ \\
\hline
$X(\mathrm{b3},\mathrm{b5},\mathrm{ns7})$ & $B(3)$ & $B(5)$ & $C_{\mathrm{ns}}^{+}(7)$ &
504 & 37 & $\geq 4.999$ & $\geq 5.25$ \\
\hline
$X(\mathrm{b3},\mathrm{b5},\mathrm{e7})$ & $B(3)$ & $B(5)$ & $G(\mathrm{e7})$ &
1008 & 73 & $\geq 9.998$ & $\geq 10.5$ \\
\hline
$X(\mathrm{s3},\mathrm{b5},\mathrm{e7})$ & $C_{\mathrm{s}}^{+}(3)$ & $B(5)$ & $G(\mathrm{e7})$ &
1512 & 153 & $\geq 14.9963$ & $\geq 15.75$ \\
\hline
\end{tabular}
\end{center}

\vspace{0.1in}

The column ``level at $p$'' ($p = 3,5,7$) means the subgroup of
$\GL_2(\F_p)$ defining a congruence subgroup of $\SL_2(\Z)$.
More precisely,
if the column ``level at $p$'' is $G_p \subset \GL_2(\F_p)$ for $p = 3,5,7$,
the corresponding modular curve is $\Gamma \backslash \mathbb{H}$,
where
\[ \mathbb{H} \coloneqq \{ \tau \in \C \mid \mathrm{Im}\,\tau > 0\} \]
is the complex upper-half plane,
and $\Gamma$ is the congruence subgroup of $\SL_2(\Z)$
defined by
\[ \Gamma \coloneqq \{ g \in \SL_2(\Z) \mid g \ (\text{mod} \ p) \in G_p \  (p = 3,5,7) \}. \]
Since the congruence subgroups we consider in this paper contain
$- I_2$, the index of congruence subgroups inside
$\SL_2(\Z)$ and $\PSL_2(\Z)$ are the same.
We calculated the genus of
$X_0(105)$,
$X(\mathrm{s3},\mathrm{b5},\mathrm{b7})$,
and $X(\mathrm{b3},\mathrm{b5},\mathrm{ns7})$
using LMFDB.
See Appendix \ref{Appendix:Calculation} for details.
We will not use the values of the genera of
$X(\mathrm{b3},\mathrm{b5},\mathrm{e7})$,
$X(\mathrm{s3},\mathrm{b5},\mathrm{e7})$
in this paper.
These values are taken from
\cite[Introduction]{Box:Modularity}.
The column ``gonality (proved)'' contains the lower bound
of the gonality
calculated by Abramovich's theorem
(see Theorem \ref{Theorem:gonalitylowerbound} below)
combined with Kim--Sarnak's lower bound of $\lambda_1$.
The column ``gonality (expected)'' contains the lower bound
calculated assuming Selberg's conjecture.

\begin{thm}[Abramovich {\cite[Theorem 0.1]{Abramovich:Gonality}}]
\label{Theorem:gonalitylowerbound}
Let $\Gamma \subset \PSL_2(\Z)$ be a finite index subgroup,
and $X_{\Gamma} \coloneqq \Gamma \backslash \mathbb{H}$
the corresponding modular curve.
Let $\lambda_1$ be the smallest positive eigenvalue of the Laplacian on $X_{\Gamma}$.
Then the gonality of $X_{\Gamma}$ is greater than or equal to
$\lambda_1/24 \cdot [\PSL_2(\Z) : \Gamma]$.
\end{thm}

\begin{rem}
The lower bound $\lambda_1 \geq 21/100$ is used in
Abramovich's paper \cite[Theorem 0.1]{Abramovich:Gonality}.
This lower bound was proved by Luo--Rudnick--Sarnak.
After the publication of \cite{Abramovich:Gonality},
Kim--Sarnak proved a better lower bound
$\lambda_1 \geq 975/4096$;
see \cite[Appendix 2, p.176]{Kim-Sarnak}.
We use Kim--Sarnak's bound to calculate the lower bound of the gonality
in the above table.
If Selberg's conjecture is true, then we have $\lambda_1 \geq 1/4$.
\end{rem}

\subsection{Proof of Theorem
\ref{MainTheorem:FinitenessLowDegreePointsModularCurves}
for $X_0(105)$}

The following result is proved in
Appendix \ref{Appendix:X0(105)}.

\begin{lem}
\label{Lemma:X0(105)}
\begin{enumerate}
\item The genus of $X_0(105)$ is $13$.
\item The genus of $X_0(105)/w_{35}$ is $3$.
\item The Mordell--Weil group $J_0(105)(\Q)$ is finite.
\end{enumerate}
\end{lem}

The following result is proved by an application of
the Castelnuovo-Severi inequality.

\begin{lem}
\label{Lemma:X0(105)MorphismtoP1}
Let $f \colon X_0(105) \to \PP^1$ be a non-constant morphism
of degree $d$ over $\C$ which does not factor through
$X_0(105) \to X_0(105)/w_{35}$.
Then we have $d \geq 8$.
\end{lem}

\begin{proof}
By Corollary \ref{Corollary:Castelnuovo-Severi:involution},
we have
\[
g(X_0(105)) = 13 \ \leq \ 2 g(X_0(105)/w_{35}) + d - 1 = d + 5.
\]
\end{proof}

We shall prove
Theorem
\ref{MainTheorem:FinitenessLowDegreePointsModularCurves}
for $X_0(105)$.
By Lemma \ref{Lemma:X0(105)MorphismtoP1},
there does not exist a non-constant morphism
$X_0(105) \to \PP^1$ of degree $5$ defined over $\C$.
By Theorem \ref{Theorem:FinitenessLowDegreePoints1},
the proof of
Theorem \ref{MainTheorem:FinitenessLowDegreePointsModularCurves}
for $X_0(105)$ is complete.
\qed

We can determine the $\Q$-gonality and the $\C$-gonality of $X_0(105)$.
Although it may be well-known to the specialists,
we give a sketch of the proof for the readers' convenience.

\begin{prop}
\label{Proposition:GonalityX0(105)}
Both of the $\Q$-gonality and the $\C$-gonality of $X_0(105)$ are $6$.
\end{prop}

\begin{proof}
It is well-known that
$X_0(105)/w_{35}$ is a non-hyperelliptic curve of genus $3$;
see \cite[Theorem 3]{Furumoto-Hasegawa}.
(See also \cite[Section 4]{Le Hung:Thesis}.
An explicit defining equation can be found in
\cite[Section 3.3]{Box:Modularity}.)
It is isomorphic to a smooth plane quartic.
Taking the projection from a $\Q$-rational point on $X_0(105)/w_{35}$
(e.g., a $\Q$-rational cusp),
we have a non-constant morphism
$X_0(105)/w_{35} \to \PP^1$ of degree $3$ over $\Q$.
Hence we have a non-constant morphism
$X_0(105) \to \PP^1$ of degree $6$ over $\Q$.

It remains to prove that there does not exist
a non-constant morphism
$f \colon X_0(105) \to \PP^1$ of degree $\leq 5$ over $\C$.
When $\deg f$ is odd, the assertion follows from 
Lemma \ref{Lemma:X0(105)MorphismtoP1}.
Assume $\deg f \in \{ 2,4 \}$.
By Lemma \ref{Lemma:X0(105)MorphismtoP1},
$f$ factors through
$X_0(105) \to X_0(105)/w_{35}$.
However, since $X_0(105)/w_{35}$ is non-hyperelliptic of genus $3$,
there does not exist a non-constant morphism
$X_0(105)/w_{35} \to \PP^1$ of degree $\leq 2$ over $\C$.
This completes the proof.
\end{proof}

\subsection{Proof of Theorem
\ref{MainTheorem:FinitenessLowDegreePointsModularCurves}
for $X(\mathrm{s3},\mathrm{b5},\mathrm{b7})$}

The following result is proved in
Appendix \ref{Appendix:X(s3b5b7)}.

\begin{lem}
\label{Lemma:X(s3b5b7)}
\begin{enumerate}
\item The genus of $X(\mathrm{s3},\mathrm{b5},\mathrm{b7})$ is $21$.
\item The genus of $X(\mathrm{s3},\mathrm{b5},\mathrm{b7})/w_{35}$ is $7$.
\item Let $f$ be a cusp form of weight $2$ corresponding to
a holomorphic differential form on $X(\mathrm{s3},\mathrm{b5},\mathrm{b7})$.
Then, either the analytic rank of $f$ is $0$,
or the Hecke field of $f$ (i.e.,\ the field generated by the Fourier coefficients of $f$) is larger than $\Q$.
\end{enumerate}
\end{lem}

The following lemma can be proved by the same way
as Lemma \ref{Lemma:X0(105)MorphismtoP1}.

\begin{lem}
\label{Lemma:X(s3b5b7)MorphismtoP1}
Let $f \colon X(\mathrm{s3},\mathrm{b5},\mathrm{b7}) \to \PP^1$
be a non-constant morphism
of degree $d$ over $\C$ which does not factor through
$X(\mathrm{s3},\mathrm{b5},\mathrm{b7}) \to X(\mathrm{s3},\mathrm{b5},\mathrm{b7})/w_{35}$.
Then we have $d \geq 8$.
\end{lem}

\begin{proof}
By Corollary \ref{Corollary:Castelnuovo-Severi:involution},
we have
\[
g(X(\mathrm{s3},\mathrm{b5},\mathrm{b7})) = 21 \ \leq
\ 2 g(X(\mathrm{s3},\mathrm{b5},\mathrm{b7})/w_{35}) + d - 1 = d + 13.
\]
\end{proof}

We shall apply Theorem \ref{Theorem:FinitenessLowDegreePoints2}
to $X(\mathrm{s3},\mathrm{b5},\mathrm{b7})$.
\begin{itemize}
\item
Since
$g(X(\mathrm{s3},\mathrm{b5},\mathrm{b7})) = 21 \geq 11$,
the condition on the genus is satisfied.

\item 
By Lemma \ref{Lemma:X(s3b5b7)} (3),
the condition on the Mordell--Weil group is satisfied;
see Remark \ref{Remark:BSDknowncases}.

\item 
By Lemma \ref{Lemma:X(s3b5b7)MorphismtoP1},
there does not exist a non-constant morphism
$X(\mathrm{s3},\mathrm{b5},\mathrm{b7}) \to \PP^1$
of degree $5$ defined over $\C$.
\end{itemize}

Hence we can apply
Theorem \ref{Theorem:FinitenessLowDegreePoints2} to
$X(\mathrm{s3},\mathrm{b5},\mathrm{b7})$.
The proof of
Theorem \ref{MainTheorem:FinitenessLowDegreePointsModularCurves}
for $X(\mathrm{s3},\mathrm{b5},\mathrm{b7})$
is complete.
\qed

\subsection{Proof of Theorem
\ref{MainTheorem:FinitenessLowDegreePointsModularCurves}
for $X(\mathrm{b3},\mathrm{b5},\mathrm{ns7})/w_3$}

The following result is proved in
Appendix \ref{Appendix:X(b3b5bns7)w3}.

\begin{lem}
\label{Lemma:X(b3b5bns7)w3}
\begin{enumerate}
\item The genus of $X(\mathrm{b3},\mathrm{b5},\mathrm{ns7})/w_3$ is $19$.
\item The genus of $X(\mathrm{b3},\mathrm{b5},\mathrm{ns7})/\langle w_3, w_5 \rangle$ is $6$.
\item Let $f$ be a cusp form of weight $2$ corresponding to
a holomorphic differential form on $X(\mathrm{b3},\mathrm{b5},\mathrm{ns7})/w_3$.
Then, either the analytic rank of $f$ is $0$,
or the Hecke field of $f$ is larger than $\Q$.
\end{enumerate}
\end{lem}

The following lemma can be proved by
the same way as
Lemma \ref{Lemma:X0(105)MorphismtoP1} and
Lemma \ref{Lemma:X(s3b5b7)MorphismtoP1}.

\begin{lem}
\label{Lemma:X(b3b5bns7)w3MorphismtoP1}
Let $f \colon X(\mathrm{b3},\mathrm{b5},\mathrm{ns7})/w_3 \to \PP^1$
be a non-constant morphism
of degree $d$ over $\C$ which does not factor through
$X(\mathrm{b3},\mathrm{b5},\mathrm{ns7})/w_3 \to 
X(\mathrm{b3},\mathrm{b5},\mathrm{ns7})/\langle w_3, w_5 \rangle$.
Then we have $d \geq 8$.
\end{lem}

\begin{proof}
By Corollary \ref{Corollary:Castelnuovo-Severi:involution},
we have
\[
g(X(\mathrm{b3},\mathrm{b5},\mathrm{ns7})/w_3) = 19 \ \leq
\ 2 g(X(\mathrm{b3},\mathrm{b5},\mathrm{ns7})/\langle w_3, w_5 \rangle) + d - 1 = d + 11.
\]
\end{proof}

We shall apply Theorem \ref{Theorem:FinitenessLowDegreePoints2}
to $X(\mathrm{b3},\mathrm{b5},\mathrm{ns7})/w_3$.
\begin{itemize}
\item 
Since
$g(X(\mathrm{b3},\mathrm{b5},\mathrm{ns7})/w_3) = 19 \geq 11$,
the condition on the genus is satisfied.

\item 
By Lemma \ref{Lemma:X(b3b5bns7)w3} (3),
the condition on the Mordell--Weil group is satisfied;
see Remark \ref{Remark:BSDknowncases}.

\item 
By Lemma \ref{Lemma:X(b3b5bns7)w3MorphismtoP1},
there does not exist a non-constant morphism
$X(\mathrm{b3},\mathrm{b5},\mathrm{ns7})/w_3 \to \PP^1$
of degree $5$ defined over $\C$.
\end{itemize}

Hence we can apply
Theorem \ref{Theorem:FinitenessLowDegreePoints2}
to $X(\mathrm{b3},\mathrm{b5},\mathrm{ns7})/w_3$.
The proof of
Theorem \ref{MainTheorem:FinitenessLowDegreePointsModularCurves}
for $X(\mathrm{b3},\mathrm{b5},\mathrm{ns7})/w_3$
is complete.
\qed

\subsection{Proof of Theorem
\ref{MainTheorem:FinitenessLowDegreePointsModularCurves}
for $X(\mathrm{s3},\mathrm{b5},\mathrm{e7})$}

This case is easy because the gonality is very large.
The gonality of $X(\mathrm{s3},\mathrm{b5},\mathrm{e7})$
is larger than or equal to $15$;
see Section \ref{Subsection:GonalityModularCurves}.
In particular, it is larger than or equal to $11$.
By Theorem \ref{Theorem:AbramovichFrey},
the set of rational points on
$X(\mathrm{s3},\mathrm{b5},\mathrm{e7})$
of degree $5$ is finite.
This proves
Theorem \ref{MainTheorem:FinitenessLowDegreePointsModularCurves}
for $X(\mathrm{s3},\mathrm{b5},\mathrm{e7})$.
\qed

Summarizing the results in this section,
the proof of Theorem \ref{MainTheorem} is complete.

\appendix

\section{Methods of calculation}
\label{Appendix:Calculation}

\subsection{Decomposition of the space of cusp forms}

The main results of this paper depend on
the calculation of the genera of some modular curves and
the decomposition of their Jacobian varieties, up to isogeny.
Here we explain how to calculate them using LMFDB.

First, we recall basic results on modular forms and modular curves.
In this Appendix,
we only consider cusp forms of weight $2$ with trivial character.
We use the notation of \cite[Section 3]{BGJGP}.
Let $N \geq 1$.
Let $S_2(\Gamma_0(N))$ be the $\C$-vector space of
holomorphic cusp forms of weight $2$, of level $N$,
and with trivial character.
For every $M \geq 1$ dividing $N$ and $d \geq 1$ dividing $N/M$,
there is a map 
\[
\iota_{M,d} \colon S_2(\Gamma_0(M)) \to S_2(\Gamma_0(N)),
\quad
f(q) \mapsto f(q^d).
\]
The subspace of $S_2(\Gamma_0(N))$ spanned by
the image of $\iota_{M,d}$ for $(M,d)$ is called
the space of \textit{old forms}.
The subspace of newforms $S_2^{\new}(\Gamma_0(N))$ is defined to be
the orthogonal complement of
the subspace of old forms with respect to the Petersson inner product.
There is a $\C$-basis 
$\New_N \subset S_2^{\new}(\Gamma_0(N))$
consisting of normalized Hecke eigenforms.
The absolute Galois group $\Gal(\overline{\Q}/\Q)$ acts on $\New_N$
via the Fourier coefficients.
Let $\Gal(\overline{\Q}/\Q) \backslash \New_N$
be the set of $\Gal(\overline{\Q}/\Q)$-orbits.
For $f = \sum_{n \geq 1} a_n q^n \in \New_N$,
let $E_f \coloneqq \Q(a_n)$ be the subfield of $\C$
generated by the Fourier coefficients.
It is a number field embedded in $\C$.
It is called the \textit{field of coefficients}
or the \textit{Hecke field} of $f$.
For an embedding $\tau \colon E_f \hookrightarrow \C$
and $d \geq 1$,
we put
$\prescript{\tau}{}f(q^d) \coloneqq \sum_{n \geq 1} \tau(a_n) q^{dn}$.

We have the following decomposition (see \cite[(3.3)]{BGJGP}):
\begin{equation}
\label{Equation:DecompositionModularForms}
S_2(\Gamma_0(N))
\ =
\ \bigoplus_{M|N}
\ \bigoplus_{f \in \Gal(\overline{\Q}/\Q) \backslash \New_M}
\ \bigoplus_{d | (N/M)}
\ \bigoplus_{\tau \colon E_f \hookrightarrow \C}
\ \C \, \prescript{\tau}{}f(q^d).
\end{equation}
We put
\[
V_f \coloneqq
\bigoplus_{d | (N/M)}
\ \bigoplus_{\tau \colon E_f \hookrightarrow \C}
\ \C \, \prescript{\tau}{}f(q^d).
\]
We call it \textit{the subspace associated with $f$}.
We have
\begin{equation}
\label{Equation:DecompositionModularForms2}
S_2(\Gamma_0(N))
\ =
\ \bigoplus_{M|N}
\ \bigoplus_{f \in \Gal(\overline{\Q}/\Q) \backslash \New_M} V_f.
\end{equation}

Let $A_f$ be the abelian variety of $\GL_2$-type over $\Q$
associated to $f$ constructed by Shimura.
It is an abelian variety of dimension $[E_f : \Q]$.
Ribet proved it is $\Q$-simple \cite[Corollary 4.2]{Ribet:twists}.
It is known that $J_0(N) \coloneqq \Jac(X_0(N))$
is isogenous to the product of several copies of $A_f$
for $f \in \New_M$.
Here $M$ is a positive integer dividing $N$.
More precisely, we have the following decomposition parallel to
(\ref{Equation:DecompositionModularForms2})
(see \cite[(3.3)]{BGJGP}):
\begin{equation}
\label{Equation:DecompositionAbelianVarieties}
J_0(N) \ \sim
\ \bigoplus_{M|N}
\ \bigoplus_{f \in \Gal(\overline{\Q}/\Q) \backslash \New_M}
\ A_f^{n_f}.
\end{equation}
This decomposition is defined over $\Q$.
The positive integer $n_f \geq 1$ is the multiplicity of $A_f$
in $J_0(N)$.
It is equal to the number of positive divisors of $M/N$.
In particular, if $M = N$, then $n_f = 1$.
We have $\dim_{\C} V_f = n_f [E_f : \Q]$.

\subsection{Action of Atkin--Lehner operators on the space of cusp forms }

Let $N \geq 1$.
Let $Q$ be a positive divisor of $N$ satsifying
$\mathrm{gcd}(Q,N/Q) = 1$.
Take $x, y, z, t \in \Z$ satisfying $Q^2 xt - N yz = Q$.
The Atkin--Lehner operator $w_Q$
is defined by the matrix
$\begin{pmatrix}
Qx & y \\ Nz & Qt
\end{pmatrix}$
of determinant $Q$.
Explicitly, we have
\[
w_Q(f)(\tau) \coloneqq
\frac{Q}{(Nz \tau + Qt)^2} \, f\bigg( \frac{Qx \tau + y}{Nz \tau + Qt} \bigg)
\]
for $f \in S_2(\Gamma_0(N))$.
Here $\tau \in \C \ (\mathrm{Im}\,\tau > 0)$
is the coordinate on the complex upper-half plane,
and $q \coloneqq \exp(2 \pi \sqrt{-1} \tau)$.
The operator $w_Q$ does not depend on the choice of $x, y, z, t$.
It is an involution on $S_2(\Gamma_0(N))$,
i.e, it satisfies $w_Q \circ w_Q = \mathrm{id}$.
Let $p$ be a prime number dividing $N$.
Let $n \coloneqq v_p(N)$ be the $p$-adic valuation of $N$,
i.e., $p^n$ divides $N$ and $p^{n+1}$ does not divide $N$.
The operator $w_{p^n}$ is sometimes written as $w_p$.
For $f \in \New_N$, we have $w_{p^n}(f) = \varepsilon_p f$
for some $\varepsilon_p \in \{ \pm 1\}$.
The sign $\varepsilon_p$ is called the \textit{Fricke sign} of $f$ at $p$.

We shall calculate the action of $w_{p^n}$
on the space of old forms.
Let $M$ be a positive divisor of $N$ and $f \in \New_M$.
If $p$ divides $M$, let $\varepsilon_p$ be the Fricke sign of $f$ at $p$.
If $p$ does not divide $M$, we put $\varepsilon_p \coloneqq 1$.
Let $d$ be the positive divisor of $N/M$ such that
$v_p(d) \leq v_p(N/M)/2$.
We put $d' \coloneqq p^{v_p(N/M) - 2 v_p(d)} \cdot d$.
Then, the calculation in the bottom of \cite[p.\ 148]{Atkin-Lehner}
shows the following.
\begin{itemize}
\item Assume $d \neq d'$.
In this case, the operator $w_{p^n}$ interchanges
the subspaces $\C f(q^d)$ and $\C f(q^{d'})$.
Hence both of the $(+1)$-eigenspace and the $(-1)$-eigenspace
are one-dimensional in $\C f(q^d) \oplus \C f(q^{d'})$.

\item Assume $d = d'$.
This occurs if and only if
$v_p(N/M)$ is even and $v_p(d) = v_p(N/M)/2$.
In this case, the operator $w_{p^n}$ stabilizes the subspace
$\C f(q^d)$ and its one-dimensional eigenvalue is $\varepsilon_p$.
\end{itemize}

Recall that the subspace
$V_f \subset S_2(\Gamma_0(N))$ associated with $f$
is spanned by
$\prescript{\tau}{}f(q^d)$
for $d | (N/M)$ and $\tau \colon E_f \hookrightarrow \C$.
Hence we have
\[
\dim V_f = [E_f : \Q] \cdot \prod_{\substack{q \, : \, \text{prime} \\ q |(N/M)}} (v_q(N/M) + 1).
\]

In summary, we have the following result.

\begin{lem}
\label{Lemma:AtkinLehnerOldForms}
Let $d_{+}$ (resp.\ $d_{-}$) be the dimension of the
$(+1)$-eigenspace (resp.\ $(-1)$-eigenspace)
of the Atkin--Lehner operator $w_{p^n}$ on $V_f$.
\begin{enumerate}
\item If $v_p(N/M)$ is odd, then
$d_{+} = d_{-} = (\dim V_f)/2$.

\item If $v_p(N/M)$ is even, then
\[
d_{\pm} = [E_f : \Q] \cdot \frac{v_p(N/M) + (1 \pm \varepsilon_p)}{2} \cdot \prod_{\substack{q \, : \, \text{prime},\ q \neq p \\ q |(N/M)}} (v_q(N/M) + 1).
\]
\end{enumerate}
\end{lem}

\subsection{Calculation for $X_0(105)$}
\label{Appendix:X0(105)}

Here we shall prove Lemma \ref{Lemma:X0(105)} using LMFDB.
Though Lemma \ref{Lemma:X0(105)} is well-known to the specialists,
we record the proof for the readers' convenience.

We shall search newforms of weight $2$, of level dividing 105,
and with trivial character.
Then we find $6$ such modular forms in total.

\vspace{0.1in}

\begin{center}
\begin{tabular}{|c|c|c|c|c|c|c|}
\hline
$f$ & $\dim A_f$
& \makecell{analytic \\ rank}
& \makecell{mult.}
& \makecell{Fricke sign \\ at $3$}
& \makecell{Fricke sign \\ at $5$}
& \makecell{Fricke sign \\ at $7$} \\
\hline
15.2.a.a & $1$ & $0$ & $2$ &  $1$ & $-1$ & \\
\hline
21.2.a.a & $1$ & $0$ & $2$ & $-1$ & & $1$ \\
\hline
35.2.a.a & $1$ & $0$ & $2$ & & $1$ & $-1$ \\
\hline
35.2.a.b  & $2$ & $0$ & $2$ & & $-1$ & $1$ \\
\hline
105.2.a.a & $1$ & $0$ & $1$ & $-1$ & $-1$ & $-1$ \\
\hline
105.2.a.b & $2$ & $0$ & $1$ & $1$ & $1$ & $-1$ \\
\hline
\end{tabular}
\end{center}

\vspace{0.1in}

\begin{quote}
\url{https://www.lmfdb.org/ModularForm/GL2/Q/holomorphic/?level_type=divides&level=105&weight=2&char_order=1}
\end{quote}

\vspace{0.1in}

Here ``$\dim A_f$'' is the dimension of the modular abelian variety
of $\GL_2$-type associated with $f$;
it is equal to the degree of the Hecke field.
The column ``mult.'' contains the multiplicity of the form
in the space of modular forms of level 105.
If $f$ is a newform of level $N_f$, the multiplicity is equal
to the number of positive divisors of $105/N_f$.

From the above list, it is easy to calculate the genera of
$X_0(105)$ and $X_0(105)/w_{35}$
and the Mordell--Weil rank of $J_0(105) = \Jac(X_0(105))$ over $\Q$.
The genus of $X_0(105)$
is equal to the sum of $\dim A_f$ times multiplicities.
Hence we have $g(X_0(105)) = 13$.
Since all cusp forms in the above list of have analytic rank $0$,
their associated abelian varieties have Mordell--Weil rank $0$ over $\Q$;
see Remark \ref{Remark:BSDknowncases}.
Therefore, $J_0(105)(\Q)$ is finite.

Next, we shall calculate the genus of the Atkin--Lehner quotient
$X_0(105)/w_{35}$.
A cusp form $f$ on $X_0(105)$ becomes
a holomorphic differential form on $X_0(105)/w_{35}$
if and only if $w_{35}(f) = f$.
The cusp form 105.2.a.a satisfies this condition,
because both of the Fricke signs at $5$ and $7$ are $-1$;
we have $w_{35}(f) = w_{5}(w_{7}(f)) = -w_{5}(f) = f$
if $f$ is the cusp form 105.2.a.a.
The cusp form 105.2.a.b does not satisfy this condition
because the Fricke sign at $5$ and $7$ are different.
Let us consider the cusp form 15.2.a.a, which is an old form.
The space of old forms associated to $f$ is $2$-dimensional.
By Lemma \ref{Lemma:AtkinLehnerOldForms},
both of the $(\pm 1)$-eigenspaces of $w_{35}$ are one-dimensional.
Hence the contribution of this modular form to
$X_0(105)/w_{35}$ is one-dimensional.
Similarly, the contribution of 21.2.a.a to $X_0(105)/w_{35}$
is one-dimensional.
Since other cusp forms do not contribute to $X_0(105)/w_{35}$,
the space of cusp form of weight 2 on
$X_0(105)/w_{35}$ is three dimensional in total.
Hence $X_0(105)/w_{35}$ has genus $3$.

This finishes the proof of  Lemma \ref{Lemma:X0(105)}.
\qed

\begin{rem}
By the above calculation,
$\Jac(X_0(105)/w_{35})$
is isogenous to the product of three elliptic curves;
they are associated with 15.2.a.a, 21.2.a.a, and 105.2.a.a.
\end{rem}

\subsection{Calculation for $X(\mathrm{s3},\mathrm{b5},\mathrm{b7})$}
\label{Appendix:X(s3b5b7)}

Although the level structure ``$\mathrm{s3}$''
associated with the normalizer of a split Cartan subgroup
is not listed by LMFDB,
we can calculate the space of cusp forms with this level structure using the isomorphism
\[
X(\mathrm{s3},\mathrm{b5},\mathrm{b7})
\cong X_0(315)/w_9.
\]
See \cite[p.\ 555]{Conrad-Diamond-Taylor} for example.
The space of cusp forms on
$X(\mathrm{s3},\mathrm{b5},\mathrm{b7})$
is identified with the $(+1)$-eigenspace of $w_9$
in the space of cusp forms on $X_0(315)$.

According to LMFDB,
there exist 15 cusp forms of weight $2$, level dividing 315,
and with trivial character.

\vspace{0.1in}

\begin{center}
\begin{tabular}{|c|c|c|c|c|c|c|}
\hline
$f$ & $\dim A_f$
& \makecell{analytic \\ rank}
& \makecell{mult.}
& \makecell{Fricke sign \\ at $3$}
& \makecell{Fricke sign \\ at $5$}
& \makecell{Fricke sign \\ at $7$} \\
\hline
15.2.a.a  & $1$ & $0$ & $4$ &  $1$ & $-1$ &      \\
\hline
21.2.a.a  & $1$ & $0$ & $4$ & $-1$ &      &  $1$ \\
\hline
35.2.a.a  & $1$ & $0$ & $3$ &      &  $1$ & $-1$ \\
\hline
35.2.a.b  & $2$ & $0$ & $3$ &      & $-1$ &  $1$ \\
\hline
45.2.a.a  & $1$ & $0$ & $2$ & $-1$ &  $1$ &      \\
\hline
63.2.a.a  & $1$ & $0$ & $2$ & $-1$ &      &  $1$ \\
\hline
63.2.a.b  & $2$ & $0$ & $2$ &  $1$ &      & $-1$ \\
\hline
105.2.a.a & $1$ & $0$ & $2$ & $-1$ & $-1$ & $-1$ \\
\hline
105.2.a.b & $2$ & $0$ & $2$ &  $1$ &  $1$ & $-1$ \\
\hline
315.2.a.a & $1$ & $1$ & $1$ & $-1$ &  $1$ & $-1$ \\
\hline
315.2.a.b & $1$ & $0$ & $1$ & $-1$ & $-1$ & $-1$ \\
\hline
315.2.a.c & $2$ & $1$ & $1$ &  $1$ &  $1$ &  $1$ \\
\hline
315.2.a.d & $2$ & $0$ & $1$ & $-1$ & $-1$ & $-1$ \\
\hline
315.2.a.e & $2$ & $0$ & $1$ & $-1$ &  $1$ &  $1$ \\
\hline
315.2.a.f & $2$ & $0$ & $1$ &  $1$ & $-1$ &  $1$ \\
\hline
\end{tabular}
\end{center}

\vspace{0.1in}

\begin{quote}
\url{https://www.lmfdb.org/ModularForm/GL2/Q/holomorphic/?level_type=divides&level=315&weight=2&char_order=1}
\end{quote}

\vspace{0.1in}

From this list,
we calculate that the genus of $X_0(315)$ is $41$.
We omit the details of the calculation because it is
the same as $X_0(105)$.
The $(+1)$-eigenspace of $w_9$ is $21$-dimensional.
Hence the genus of
$X(\mathrm{s3},\mathrm{b5},\mathrm{b7})$ is $21$.
We can similarly calculate that
the $(+1)$-eigenspace of $w_{35}$ is $7$ dimensional.
Hence the quotient
$X(\mathrm{s3},\mathrm{b5},\mathrm{b7})/w_{35}$ has genus $7$.

Let us calculate the analytic rank of $\Q$-simple factors
of $\Jac(X(\mathrm{s3},\mathrm{b5},\mathrm{b7}))$.
In the above list,
only 315.2.a.a and 315.2.a.c have positive analytic rank.
Both are newforms of level 315.
Since 315.2.a.a has $w_9$-eigenvalue $-1$,
it does not contribute to
$X(\mathrm{s3},\mathrm{b5},\mathrm{b7})$.
On the other hand, 315.2.a.c contribute to
$X(\mathrm{s3},\mathrm{b5},\mathrm{b7})$,
but the Hecke field of 315.2.a.c is $\Q(\sqrt{2})$.

The proof of Lemma \ref{Lemma:X(s3b5b7)} is complete.
\qed

\begin{rem}
Among the $\Q$-simple factors of 
$\Jac(X(\mathrm{s3},\mathrm{b5},\mathrm{b7}))$,
only one abelian variety has positive Mordell--Weil rank over $\Q$.
It is the abelian surface of $\GL_2$-type associated with 315.2.a.c.
Since 315.2.a.c has analytic rank $1$,
the Mordell--Weil rank of
$\Jac(X(\mathrm{s3},\mathrm{b5},\mathrm{b7}))$
over $\Q$ is $[\Q(\sqrt{2}) : \Q] = 2$;
see Remark \ref{Remark:BSDknowncases}.
\end{rem}

\subsection{Calculation for $X(\mathrm{b3},\mathrm{b5},\mathrm{ns7})/w_3$}
\label{Appendix:X(b3b5bns7)w3}

This case is more complicated than other cases.

First, we shall determine the $\Q$-simple factors of
the Jacobian of $X(\mathrm{b3},\mathrm{b5},\mathrm{ns7})$.
We shall use Le Hung's result
\cite[Lemma 4.2]{Le Hung:Thesis},
which is an adelic reformulation of the results of
Chen, de Smit, Edixhoven;
see also \cite[Theorem 1]{Chen:PLMS}, \cite[Theorem 1]{Chen:JAlg}.
Let $p = 7$.
Let $K^7$ be the adelic level subgroup outside $7$
corresponding to the congruence subgroup
\[
  \Gamma_0(15) =
  \{ g \in \SL_2(\Z) \mid g \ (\text{mod} \ 3) \in B(3),\ g   \ (\text{mod} \ 5) \in B(5) \}.
\]
The adelic level subgroups
$K_0(7)$, $K(7 \mathrm{s}^{+})$, and $K(7 \mathrm{ns}^{+})$
in \cite[Lemma 4.2]{Le Hung:Thesis}
correspond to the subgroups
$B(7)$, $C_\mathrm{s}^{+}(7)$, and $C_{\mathrm{ns}}^{+}(7)$
in Section \ref{Subsection:GonalityModularCurves},
respectively.
In the notation of this paper,
the isogeny stated in \cite[Lemma 4.2]{Le Hung:Thesis}
becomes
\begin{equation}
\label{Isogeny:Chen}
\Jac(X(\mathrm{b3},\mathrm{b5},\mathrm{ns7})) \times
J_0(105)
\ \sim
\ \Jac(X(\mathrm{b3},\mathrm{b5},\mathrm{s7})) \times
J_0(15).
\end{equation}

The isogeny (\ref{Isogeny:Chen})
is equivariant with respect to the action of Hecke operators
outside $7$ and the Atkin--Lehner operators $w_3$, $w_5$. 
Therefore, it is enough to determine the $\Q$-simple factors of
$J_0(105)$,
$\Jac(X(\mathrm{b3},\mathrm{b5},\mathrm{s7}))$,
and $J_0(15)$.
The determination of $\Q$-simple factors of $J_0(105)$
follows from the results in Section \ref{Appendix:X0(105)}.
The case for $J_0(15)$ is similar.
We shall determine the $\Q$-simple factors of
$\Jac(X(\mathrm{b3},\mathrm{b5},\mathrm{s7}))$
by the same method as in
Section \ref{Appendix:X(s3b5b7)}.
We shall use an isomorphism
\[
X(\mathrm{b3},\mathrm{b5},\mathrm{s7})
\ \cong \ 
X(\mathrm{b3},\mathrm{b5},\mathrm{b49})/w_{49}.
\ \cong \ 
X_0(735)/w_{49}.
\]
The action of $w_{49}$ on the space of old forms
on $X(\mathrm{b3},\mathrm{b5},\mathrm{s7})$
is calculated by Lemma \ref{Lemma:AtkinLehnerOldForms}.

We summarize the results of our calculation.
According to LMFDB, there are $35$ cusp forms of weight $2$,
of level dividing $735$, and with trivial character.
To see the list, visit the following website:
\begin{quote}
\url{https://www.lmfdb.org/ModularForm/GL2/Q/holomorphic/?level_type=divides&level=735&weight=2&char_order=1}
\end{quote}

Among them, $18$ cusp forms contribute to
$\Q$-simple factors of at least one of
$J_0(105)$,
$\Jac(X(\mathrm{b3},\mathrm{b5},\mathrm{s7}))$, and 
$J_0(15)$.
Here is the list of such forms.
In the following,
``mult.\ in $X$'' means the multiplicity of the $\Q$-simple factor
appearing in the Jacobian of the curve $X$.
The multiplicity in
$X(\mathrm{b3},\mathrm{b5},\mathrm{ns7})$
is calculated by the isogeny (\ref{Isogeny:Chen}).
From the isogeny (\ref{Isogeny:Chen}),
we have
\begin{align*}
&\quad (\text{mult.\ in \ $X(\mathrm{b3},\mathrm{b5},\mathrm{ns7})$}) \\
&=
(\text{mult.\ in \ $X(\mathrm{b3},\mathrm{b5},\mathrm{s7})$})
+ (\text{mult.\ in \ $X_0(15)$})
- (\text{mult.\ in \ $X_0(105)$}).
\end{align*}

\vspace{0.1in}

\begin{center}
\begin{tabular}{|c|c|c|c|c|c|c|}
\hline
$f$ & $\dim A_f$
& \makecell{analytic \\ rank}
& \makecell{mult.\ in \\ $X(\mathrm{b3},\mathrm{b5},\mathrm{s7})$}
& \makecell{mult.\ in \\ $X_0(15)$}
& \makecell{mult.\ in \\ $X_0(105)$}
& \makecell{mult.\ in \\ $X(\mathrm{b3},\mathrm{b5},\mathrm{ns7})$}
\\
\hline
15.2.a.a & $1$ & $0$ & $2$ & $1$ & $2$ & $1$ \\
\hline
21.2.a.a & $1$ & $0$ & $2$ & & $2$ & $0$ \\
\hline
35.2.a.a & $1$ & $0$ & $2$ & & $2$ & $0$ \\
\hline
35.2.a.b & $2$ & $0$ & $2$ & & $2$ & $0$ \\
\hline
105.2.a.a & $1$ & $0$ & $1$ & & $1$ & $0$ \\
\hline
105.2.a.b & $2$ & $0$ & $1$ & & $1$ & $0$ \\
\hline
147.2.a.c & $1$ & $0$ & $2$ & & & $2$ \\
\hline
147.2.a.d & $2$ & $1$ & $2$ & & & $2$ \\
\hline
147.2.a.e & $2$ & $0$ & $2$ & & & $2$ \\
\hline
245.2.a.e & $2$ & $1$ & $2$ & & & $2$ \\
\hline
245.2.a.f & $2$ & $0$ & $2$ & & & $2$ \\
\hline
245.2.a.h & $2$ & $0$ & $2$ & & & $2$ \\
\hline
735.2.a.b & $1$ & $1$ & $1$ & & & $1$ \\
\hline
735.2.a.e & $1$ & $0$ & $1$ & & & $1$ \\
\hline
735.2.a.g & $2$ & $0$ & $1$ & & & $1$ \\
\hline
735.2.a.i & $2$ & $1$ & $1$ & & & $1$ \\
\hline
735.2.a.n & $4$ & $0$ & $1$ & & & $1$ \\
\hline
735.2.a.o & $4$ & $0$ & $1$ & & & $1$ \\
\hline
\end{tabular}
\end{center}

\vspace{0.1in}

From the above list, we calculate that
the genus of $X(\mathrm{b3},\mathrm{b5},\mathrm{ns7})$ is $37$;
it is equal to the sum of
``$\dim A_f$'' $\cdot$
``mult.\ in $X(\mathrm{b3},\mathrm{b5},\mathrm{ns7})$.''
Also, we calculate that the Mordell--Weil rank of
$\Jac(X(\mathrm{b3},\mathrm{b5},\mathrm{ns7}))$ over $\Q$
is $11$; see Remark \ref{Remark:BSDknowncases}.

Next, we calculate the dimension of the $(+1)$-eigenspace of $w_3$
on the space of cusp forms of weight $2$ on
$X(\mathrm{b3},\mathrm{b5},\mathrm{ns7})$.
It turns out that, among the $18$ cusp form in the above list,
the following $8$ cusp forms contribute to
$X(\mathrm{b3},\mathrm{b5},\mathrm{ns7})/w_3$.

\vspace{0.1in}

\begin{center}
\begin{tabular}{|c|c|c|c|c|c|c|}
\hline
$f$ & $\dim A_f$
& \makecell{analytic \\ rank}
& \makecell{multiplicity}
& \makecell{Fricke sign \\ at $3$}
& \makecell{Fricke sign \\ at $5$}
& \makecell{Fricke sign \\ at $7$} \\
\hline
15.2.a.a  & $1$ & $0$ & $1$ & $1$ & $-1$ & \\
\hline
147.2.a.d & $2$ & $1$ & $2$ & $1$ &  & $1$ \\
\hline
245.2.a.e & $2$ & $1$ & $1$ &     & $1$ & $1$ \\
\hline
245.2.a.f & $2$ & $0$ & $1$ &     & $-1$ & $1$ \\
\hline
245.2.a.h & $2$ & $0$ & $1$ &     & $-1$ & $1$ \\
\hline
735.2.a.g & $2$ & $0$ & $1$ & $1$ & $-1$ & $1$ \\
\hline
735.2.a.i & $2$ & $1$ & $1$ & $1$ & $1$ & $1$ \\
\hline
735.2.a.n & $4$ & $0$ & $1$ & $1$ & $-1$ & $1$ \\
\hline
\end{tabular}
\end{center}

\vspace{0.1in}

We see that the only one-dimensional $\Q$-simple factor of
$\Jac(X(\mathrm{b3},\mathrm{b5},\mathrm{ns7})/w_3)$
is the modular elliptic curve associated with 15.2.a.a;
it has analytic rank $0$.

As before, we calculate that
the genus of the quotient
$X(\mathrm{b3},\mathrm{b5},\mathrm{ns7})/w_3$
is $19$.
The Mordell--Weil rank of
$\Jac(X(\mathrm{b3},\mathrm{b5},\mathrm{ns7})/w_3)$
over $\Q$ is $8$.
The $\Q$-simple factors
with positive Mordell--Weil rank over $\Q$
are the abelian varieties of $\GL_2$-type
associated with 147.2.a.d, 245.2.a.e, 735.2.a.i.
The multiplicities are $2, 1, 1$, respectively.
The multiplicities are calculated by
Lemma \ref{Lemma:AtkinLehnerOldForms}.
All of their Hecke fields are $\Q(\sqrt{2})$.

Finally, we calculate the dimension of the $(+1)$-eigenspace of $w_5$
on the space of cusp forms of weight $2$ on
$X(\mathrm{b3},\mathrm{b5},\mathrm{ns7})/w_3$.
The cusp forms contributing to the quotient
$X(\mathrm{b3},\mathrm{b5},\mathrm{ns7})/\langle w_3, w_5 \rangle$
are 147.2.a.d, 245.2.a.e, 735.2.a.i.
All of their multiplicities are $1$
by Lemma \ref{Lemma:AtkinLehnerOldForms}.
Since all of their Hecke fields are $\Q(\sqrt{2})$,
we calculate that the genus of
$X(\mathrm{b3},\mathrm{b5},\mathrm{ns7})/\langle w_3, w_5 \rangle$
is $6$.
Moreover, the Mordell--Weil rank of
$\Jac(X(\mathrm{b3},\mathrm{b5},\mathrm{ns7})/\langle w_3, w_5 \rangle)$
over $\Q$ is $6$.

The proof of Lemma \ref{Lemma:X(b3b5bns7)w3} is complete.
\qed

\begin{rem}
\label{Remark:WhyAtkinLehnerquotient}
In this Remark and the next Remark,
we explain why we chose the Atkin--Lehner quotient
$X(\mathrm{b3},\mathrm{b5},\mathrm{ns7})/w_3$.
According to LMFDB,
the newform 735.2.a.b has analytic rank $1$ and
Hecke field $\Q$.
The eigenvalues of $w_3$, $w_5$, $w_{49}$
are $-1$, $-1$, $+1$, respectively.
It is contained in the $(+1)$-eigenspace of $w_{49}$.
Hence the modular elliptic curve associated with
735.2.a.b has positive Mordell--Weil rank over $\Q$,
and appears in
$\Jac(X(\mathrm{b3},\mathrm{b5},\mathrm{ns7}))$
as a $\Q$-simple factor.
Since
$X(\mathrm{b3},\mathrm{b5},\mathrm{e7})
\to X(\mathrm{b3},\mathrm{b5},\mathrm{ns7})$
is a double covering,
this elliptic curve also appears in
$\Jac(X(\mathrm{b3},\mathrm{b5},\mathrm{e7}))$
as a $\Q$-simple factor.
Therefore, we cannot directly apply
Theorem \ref{Theorem:FinitenessLowDegreePoints2} to
$X(\mathrm{b3},\mathrm{b5},\mathrm{ns7})$ or
$X(\mathrm{b3},\mathrm{b5},\mathrm{e7})$.
According to LMFDB,
735.2.a.b is the only cusp form of positive analytic rank
with Hecke field $\Q$ contributing to
$X(\mathrm{b3},\mathrm{b5},\mathrm{ns7})$.
Since the Fricke sign of 735.2.a.b at $3$ is $-1$,
it does not contribute to
$X(\mathrm{b3},\mathrm{b5},\mathrm{ns7})/w_3$.
This is the reason why we can apply
Theorem \ref{Theorem:FinitenessLowDegreePoints2}
to this quotient.
\end{rem}

\begin{rem}
Here we explain why we do not consider the quotient
$X(\mathrm{b3},\mathrm{b5},\mathrm{ns7})/w_5$.
Since the Fricke sign of 735.2.a.b at $5$ is $-1$,
it does not contribute to
$X(\mathrm{b3},\mathrm{b5},\mathrm{ns7})/w_5$.
The condition on the Mordell--Weil group is satisfied
for this quotient.
However, since
$X(\mathrm{b3},\mathrm{b5},\mathrm{ns7})/w_5$
has genus $16$ and its quotient
$X(\mathrm{b3},\mathrm{b5},\mathrm{ns7})/\langle w_3, w_5 \rangle$
has genus $6$,
Corollary \ref{Corollary:Castelnuovo-Severi:involution} only gives
the following estimate
\[
g(X(\mathrm{b3},\mathrm{b5},\mathrm{ns7})/w_5) = 16 \ \leq
\ 2 g(X(\mathrm{b3},\mathrm{b5},\mathrm{ns7})/\langle w_3, w_5 \rangle) + d - 1 = d + 11.
\]
From this, we cannot conclude
the non-existence of a degree $5$ morphism to $\PP^1$
by the same method as before.
Although there is a possibility that
a sharper lower bound of the gonality of
$X(\mathrm{b3},\mathrm{b5},\mathrm{ns7})/w_5$
may be obtained by other methods,
we do not currently know how to apply
Theorem \ref{Theorem:FinitenessLowDegreePoints2} to
$X(\mathrm{b3},\mathrm{b5},\mathrm{ns7})/w_5$.
\end{rem}

\subsection*{Acknowledgements}

The work of Y.I. was supported by
JSPS KAKENHI Grant Number 21K13773.
The work of T.I. was supported by JSPS KAKENHI Grant Number
21H00973.
The work of Y.I. and T.I. was supported by
JSPS KAKENHI Grant Number 21K18577.
The work of S.Y. was supported by JSPS KAKENHI Grant Number
19K14514.
The authors of this paper would like to thank
Masao Oi and Masataka Chida for invaluable discussions
on modular forms and the Mordell--Weil groups.
The authors would also like to thank
the referees for reading the first draft of this
paper carefully and giving
helpful comments and suggestions.
The results of this paper cannot be obtained without use of
\textit{the L-functions and modular forms database (LMFDB)}
\cite{LMFDB}.
The authors would like to thank all the people working
for this wonderful database.

\bibliographystyle{amsplain}

\begin{thebibliography}{99}

\bibitem{Abramovich:Gonality}
  Abramovich, D.,
  \textit{A linear lower bound on the gonality of modular curves},
  Internat.\ Math.\ Res.\ Notices 1996, no.\ 20, 1005-1011.

\bibitem{Abramovich-Harris}
  Abramovich, D., Harris, J.,
  \textit{Abelian varieties and curves in $W_d(C)$},
  Compositio Math.\ 78 (1991), no.\ 2, 227-238.

\bibitem{Accola}
  Accola, R.\ D.\ M.,
  \textit{Topics in the theory of Riemann surfaces},
  Lecture Notes in Mathematics, 1595.\ Springer-Verlag, Berlin, 1994.

\bibitem{ACGH}
  Arbarello, E., Cornalba, M., Griffiths, P.\ A., Harris, J.,
  \textit{Geometry of algebraic curves}, Vol.\ I,
  Grundlehren der Mathematischen Wissenschaften, vol.\ 267.\ Springer-Verlag, New York, 1985.

\bibitem{Atkin-Lehner}
  Atkin, A.\ O.\ L., Lehner, J.,
  \textit{Hecke operators on $\Gamma_0(m)$},
  Math.\ Ann.\ 185 (1970), 134-160.

\bibitem{BGJGP}
  Baker, M.\ H., Gonz\'alez-Jim\'enez, E.,
  Gonz\'alez, J., Poonen, B.,
  \textit{Finiteness results for modular curves of genus at least 2},
  Amer.\ J.\ Math.\ 127 (2005), no.\ 6, 1325-1387.

\bibitem{Biswas}
  Biswas, I., \textit{On the Hodge cycles of Prym varieties},
  Ann.\ Sc.\ Norm.\ Super.\ Pisa Cl.\ Sci.\ (5) 3 (2004), no.\ 3, 625-635.


\bibitem{Bosch-Luetkebohmert-Raynaud}
  Bosch, S., L\"utkebohmert, W., Raynaud, M.,
  \textit{N\'eron models},
  Ergebnisse der Mathematik und ihrer Grenzgebiete (3),
  21.\ Springer-Verlag, Berlin, 1990.


\bibitem{Box:Modularity}
  Box, J.,
  \textit{Elliptic curves over totally real quartic fields not containing $\sqrt{5}$ are modular},
  Trans.\ Amer.\ Math.\ Soc.\ 375 (2022), no.\ 5, 3129-3172.

\bibitem{Breuil-Conrad-Diamond-Taylor}
  Breuil, C., Conrad, B., Diamond, F., Taylor, R.,
  \textit{On the modularity of elliptic curves over Q: wild $3$-adic exercises},
  J.\ Amer.\ Math.\ Soc.\ 14 (2001), no.\ 4, 843-939.

\bibitem{Chen:PLMS}
  Chen, I.,
  \textit{The Jacobians of non-split Cartan modular curves},
  Proc.\ London Math.\ Soc.\ (3) 77 (1998), no.\ 1, 1-38.

\bibitem{Chen:JAlg}
  Chen, I.,
  \textit{On relations between Jacobians of certain modular curves},
  J.\ Algebra 231 (2000), no.\ 1, 414-448.

\bibitem{Conrad-Diamond-Taylor}
  Conrad, B., Diamond, F., Taylor, R.,
  \textit{Modularity of certain potentially Barsotti-Tate
  Galois representations},
  J.\ Amer.\ Math.\ Soc.\ 12 (1999), no.\ 2, 521-567.

\bibitem{Debarre-Fahlaoui}
  Debarre, O., Fahlaoui, R.,
  \textit{Abelian varieties in $W^r_d(C)$ and points of bounded degree on algebraic curves},
  Compositio Math.\ 88 (1993), no.\ 3, 235-249.

\bibitem{Debarre-Klassen}
  Debarre, O., Klassen, M.\ J.,
  \textit{Points of low degree on smooth plane curves},
  J.\ Reine Angew.\ Math.\ 446 (1994), 81-87.

\bibitem{Derickx-Sutherland}
  Derickx, M., Sutherland, A.\ V.,
  \textit{Torsion subgroups of elliptic curves over quintic and sextic number fields},
  Proc.\ Amer.\ Math.\ Soc.\ 145 (2017), no.\ 10, 4233-4245.

\bibitem{Derickx-Najman-Siksek}
  Derickx, M., Najman, F., Siksek, S.,
  \textit{Elliptic curves over totally real cubic fields are modular},
  Algebra Number Theory \textbf{14} (2020), no.\ 7, 1791-1800.

\bibitem{Faltings:Lang}
  Faltings, G.,
  \textit{The general case of S.\ Lang's conjecture},
  Barsotti Symposium in Algebraic Geometry (Abano Terme, 1991),
  175-182,
  Perspect.\ Math., 15, Academic Press, San Diego, CA, 1994.

\bibitem{Freitas-Le Hung-Siksek}
  Freitas, N., Le Hung, B.\ V., Siksek, S.,
  \textit{Elliptic curves over real quadratic fields are modular},
  Invent.\ Math.\ 201 (2015), no.\ 1, 159-206.

\bibitem{Frey}
  Frey, G.,
  \textit{Curves with infinitely many points of fixed degree},
  Israel J.\ Math.\ 85 (1994), no.\ 1-3, 79-83.

\bibitem{Furumoto-Hasegawa}
  Furumoto, M., Hasegawa, Y.,
  \textit{Hyperelliptic quotients of modular curves $X_0(N)$},
  Tokyo J.\ Math.\ 22 (1999), no.\ 1, 105-125.

\bibitem{vanderGeer:points}
  Geer, G.\ van der,
  \textit{Points of degree $d$ on curves over number fields},
  Diophantine approximation and abelian varieties (Soesterberg, 1992), 111-116, Lecture Notes in Math., 1566, Springer, Berlin, 1993.

\bibitem{Harris-Silverman}
  Harris, J., Silverman, J.,
  \textit{Bielliptic curves and symmetric products},
  Proc.\ Amer.\ Math.\ Soc.\ 112 (1991), no.\ 2, 347-356.

\bibitem{Kalyanswamy}
  Kalyanswamy, S.,
  \textit{Remarks on automorphy of residually dihedral representations},
  Math.\ Res.\ Lett.\ 25 (2018), no.\ 4, 1285-1304.

\bibitem{Kato:EulerSystem}
  Kato, K.,
  \textit{$p$-adic Hodge theory and values of zeta functions of modular forms},
  Cohomologies $p$-adiques et applications arithm\'etiques.\ III.\ Ast\'erisque No.\ 295 (2004), ix, 117-290.

\bibitem{Kim-Sarnak}
  Kim, H.\ H.,
  \textit{Functoriality for the exterior square of $\mathrm{GL}_4$ and the symmetric fourth of $\mathrm{GL}_2$},
  With appendix 1 by Dinakar Ramakrishnan and appendix 2 by Kim and Peter Sarnak,
  J.\ Amer.\ Math.\ Soc.\ 16 (2003), no.\ 1, 139-183.

\bibitem{Le Hung:Thesis}
  Le Hung, B.\ V., \textit{Modularity of some elliptic curves over totally
real fields}, Ph.D.\ Thesis, Harvard University, 2014.
  (\url{https://dash.harvard.edu/handle/1/12269826})

\bibitem{Maulik-Poonen}
  Maulik, D., Poonen, B.,
  \textit{N\'eron-Severi groups under specialization},
  Duke Math.\ J.\ 161 (2012), no.\ 11, 2167-2206.


\bibitem{Ribet:twists}
  Ribet, K.\ A.,
  \textit{Twists of modular forms and endomorphisms of abelian varieties},
  Math.\ Ann.\ 253 (1980), no.\ 1, 43-62.

\bibitem{Stichtenoth}
  Stichtenoth, H.,
  \textit{Algebraic function fields and codes},
  Second edition.\ Graduate Texts in Mathematics, 254.\ Springer-Verlag, Berlin, 2009.

\bibitem{LMFDB}
  The LMFDB Collaboration, \textit{The $L$-functions and modular forms database}, 2021. \\
  (\url{https://www.lmfdb.org/})

\bibitem{Thorne}
  Thorne, J.\ A.,
  \textit{Automorphy of some residually dihedral Galois representations},
  Math.\ Ann.\ 364 (2016), no. 1-2, 589-648.

\bibitem{Thorne:Cyclotomic}
  Thorne, J.\ A.,
  \textit{Elliptic curves over $\mathbb{Q}_{\infty}$ are modular},
  J.\ Eur.\ Math.\ Soc.\ (JEMS) 21 (2019), no.\ 7, 1943-1948.

\bibitem{Tucker}
  Tucker, T.\ J.,
  \textit{Irreducibility, Brill--Noether loci and Vojta's inequality},
  With an appendix by Olivier Debarre,
  Trans.\ Amer.\ Math.\ Soc.\ 354 (2002), no.\ 8, 3011-3029.

\bibitem{Yoshikawa2}
  Yoshikawa, S.,
  \textit{On the modularity of elliptic curves over a composite field of some real quadratic fields},
  Res.\ Number Theory 2 (2016), Paper No.\ 31, 6 pp.

\bibitem{Yoshikawa1}
  Yoshikawa, S.,
  \textit{Modularity of elliptic curves over abelian totally real fields unramified at 3, 5, and 7},
  J.\ Th\'eor.\ Nombres Bordeaux 30 (2018), no.\ 3, 729-741.

\bibitem{Zhang:HeegnerPoints}
  Zhang, S.-W., \textit{Heights of Heegner points on Shimura curves},
  Ann.\ of Math.\ (2) 153 (2001), no.\ 1, 27-147.

\bibitem{Zhang:Survey}
  Zhang, S.-W.,
  \textit{Elliptic curves, $L$-functions, and CM-points},
  Current developments in mathematics, 2001, 179-219,
  Int.\ Press, Somerville, MA, 2002.

\end{thebibliography}

\end{document}